\documentclass{amsproc}
\usepackage{amsmath,amsthm,amscd,amssymb}
\usepackage{geometry}                % See geometry.pdf to learn the layout options. There are lots.
\numberwithin{equation}{section}

\theoremstyle{plain}
\newtheorem{theorem}{Theorem}[section]
\newtheorem{lemma}[theorem]{Lemma}
\newtheorem{corollary}[theorem]{Corollary}
\newtheorem{proposition}[theorem]{Proposition}

\theoremstyle{definition}
\newtheorem{definition}[theorem]{Definition}

\newtheorem{case[theorem]}{Case}

\theoremstyle{remark}
\newtheorem{remark}[theorem]{Remark}

\numberwithin{equation}{section}

\def\bb #1{ {\mathbb #1} }

\begin{document}

\title{\parbox{14cm}{\centering{Group actions and geometric combinatorics in ${\mathbb F}_q^d$}}}


\author{M. Bennett, D. Hart, A. Iosevich, J. Pakianathan and M. Rudnev}

\date{today}

\address{Department of Mathematics, University of Rochester, Rochester, NY}

\email{bennett@math.rochester.edu}
\email{iosevich@math.rochester.edu}
\email{jonathan.pakianathan@math.rochester.edu}

\address{Rockhurst University}

\email{dnh0a9@gmail.com}

\address{University of Bristol} 

\email{m.rudnev@bristol.ac.uk}

\thanks{The work of the second listed author was partially supported by the NSF Grant DMS10-45404}

\begin{abstract} In this paper we apply a group action approach to the study of Erd\H os-Falconer type problems in vector spaces over finite fields and use it to obtain non-trivial exponents for the distribution of simplices. We prove that there exists $s_0(d)<d$ such that if $E \subset {\mathbb F}_q^d$, $d \ge 2$, with $|E| \ge Cq^{s_0}$, then $|T^d_d(E)| \ge C'q^{d+1 \choose 2}$, where $T^d_k(E)$ denotes the set of congruence classes of $k$-dimensional simplices determined by $k+1$-tuples of points from $E$. Non-trivial exponents were previously obtained by Chapman, Erdogan, Hart, Iosevich and Koh (\cite{CEHIK12}) for $T^d_k(E)$ with $2 \leq k \leq d-1$. A non-trivial result for $T^2_2(E)$ in the plane was obtained by Bennett, Iosevich and Pakianathan (\cite{BIP12}). These results are significantly generalized and improved in this paper. In particular, we establish the Wolff exponent $\frac{4}{3}$, previously established in \cite{CEHIK12} for the $q\equiv3\mbox{ mod }4$ case to the case $q\equiv1\mbox{ mod }4$, and this results in a new sum-product type inequality. We also obtain non-trivial results for subsets of the sphere in ${\mathbb F}_q^d$, where previous methods have yielded nothing. The key to our approach is a group action perspective which quickly leads to natural and effective formulae in the style of the classical Mattila integral from geometric measure theory.

\end{abstract}

\maketitle

\tableofcontents

\section{Introduction}

\vskip.125in

This paper is about a class of Erd\H os-Falconer type problems in vector spaces over finite fields. These are problems that ask for the minimal size of the set $E \subset {\mathbb F}_q^d$ that guarantees that the set of values of a sufficiently non-trivial map
$$\Lambda: E^{k+1} \equiv E \times E \times \dots \times E \to {\mathbb F}_q^{l},$$
%$1 \leq k \leq d$, $1 \leq l \leq {k+1 \choose 2}$,
contains a positive proportion of, or in some cases all the possible values, in ${\mathbb F}_q^{l}$. One can also simply ask for the lower bound on $\Lambda(E \times E \times \dots \times E)$ in terms of the size $E$.

Perhaps the most celebrated of these problems is the Erd\H os-Falconer distance problem in ${\mathbb F}_q^d$, introduced by Bourgain, Katz and Tao in \cite{BKT04} in the case when $d=2$ and $q$ is prime and studied in higher dimensions over general fields by Iosevich and Rudnev (\cite{IR07}) and others. In this case, $k=1$ and $\Lambda: E \times E \to {\mathbb F}_q$ is given by

\begin{equation} \label{defimama} \Lambda(x,y)=||x-y|| \equiv {(x_1-y_1)}^2+\dots+{(x_d-y_d)}^2. \end{equation}

This mapping is not a norm since we do not impose any metric structure on ${\mathbb F}_q^d$, but it does share an important and interesting feature of the Euclidean norm: it is invariant under orthogonal matrices. This simple observation will play a key role in this paper.

The best known results on the Erd\H os-Falconer distance problem in finite fields are due to the third and fifth listed authors in \cite{IR07} in higher dimensions and to Chapman, Erdogan, Koh and the 2nd, 3rd and 4th listed authors of this paper (\cite{CEHIK12}). The former proved that if the size of $E$ is greater than $2q^{\frac{d+1}{2}}$, then $\Delta(E)={\Bbb F}_q$, while the latter showed that if $d=2$, $q \equiv 3 \bmod 4$, and the size of $E$ is greater than $Cq^{\frac{4}{3}}$, then the size of $\Delta(E)$ is greater than $cq$. In this paper, we shall prove this result without any restrictions on $q$.

In \cite{IR07} the authors proved that one cannot in general obtain $cq$ distances if the size of $E$ is $\leq q^{\frac{d}{2}}$. In \cite{HIKR11}, the size of the critical exponent was increased to $\frac{d+1}{2}$ in odd dimensions, thus showing that the result in \cite{IR07} is best possible in this setting. The question of whether the $\frac{4}{3}$ exponent in two dimensions or the exponent $\frac{d+1}{2}$ in higher even dimensions can be reduced is still open.

The Erd\H os-Falconer problem in vector spaces over finite fields may be viewed as an amalgam of the Erd\H os distance problem, which asks for the smallest number of distinct distances determined by a finite subset of ${\mathbb R}^d$, $d \ge 2$, and the Falconer distance problem (see e.g. \cite{Erd05}) which asks for the smallest Hausdorff dimension of a subset $E$ of ${\mathbb R}^d$, $d \ge 2$, such that the Lebesgue measure of the distance set
$$\Delta(E)=\{|x-y|: x,y \in E\}$$ is positive.

The Erd\H os distance problem in the plane was recently solved by Guth and Katz (\cite{GK10}). In higher dimensions the conjecture is still open, with the best results due to Solymosi and Vu (\cite{SV04}). The Falconer distance conjecture is still open, with the best results due to Wolff in two dimensions and Erdogan in higher dimensions. See \cite{CEHIK12} and the references contained therein for the best known exponents in the finite field setting. 

A natural generalization of the distance problem in vector spaces over finite fields is the problem of distribution of simplices. Before we describe the necessary definitions, we note that most of the results of this paper work with respect to any non-degenerate quadratic form $Q$. Such forms as weighted dot products yield different notions of distances. The quadratic form $Q$ has an associated linear isometry group $O(Q)$ called the orthogonal group associated to $Q$. Basic information on such forms can be found in the appendix.

\begin{definition}\label{congdef} Let $Q$ be a non-degenerate quadratic form on $\mathbb{F}_q^d$. We say that two $k$-simplices in $\mathbb{F}_q^d$ with vertices $(x_1, . . ., x_{k+1}), (y_1, . . ., y_{k+1})$ are in the same congruence class if there exists
$\theta \in O(Q)$ and $z \in \mathbb{F}_q^d$ so that $z + \theta(x_i) = y_i$ for $i = 1, 2, . . . , k+1$, or equivalently that there is $g \in \mathbb{F}_q^d \rtimes O(Q)$,
the ``Euclidean" group of motions over $\mathbb{F}_q$, such that $g(x_i)=y_i$ for all $1 \leq i \leq k+1$.  For non-degenerate simplices (see comment below and appendix), this is equivalent to saying that $\| x_i - x_j \|=\| y_i - y_j \|$ for all $1 \leq i < j \leq k+1$. Thus each (non-degenerate)
congruence class is determined by an {\bf ordered} tuple of distances which can be encoded by a symmetric $(k+1) \times (k+1)$ distance matrix
$\mathbb{D}$ with entries $d_{ij} =\| x_i - x_j \|$.

Henceforth we use the notation $O_d(\mathbb F_q)$ for the group of
 orthogonal matrices with elements in $\mathbb F_q$, with
respect to the standard dot product. We denote the set of congruence classes of $k$-simplices determined by $E \subset {\mathbb F}_q^d$ by $T^d_{k,Q}(E)$.
We use $T^d_k(E)$  in the case of the standard dot product.
\end{definition}

\begin{remark} Recall a $k$-simplex is called non-degenerate if its vertices $x_0, \dots, x_k$ are affinely independent, i.e. $x_1-x_0, x_2-x_0, \dots, x_k-x_0$ are linearly independent. The fact that two such $k$-simplices are congruent if and only if they have the same ordered set of distances is well-known and a proof can be found in \cite{CEHIK12}. If the $k$-simplices are degenerate, this is no longer true as, for example, there are $2$-simplices (triangles) with all pair distances zero which are collinear and also ones which are not and these two types cannot be congruent. However the set of degenerate $k$-simplices and their ordered distances is $o(\text{set of all possible congruent types})$ when $k \leq d$ and so this subtlety does not effect the arguments made in this paper in any significant way - see the appendix for more details. \end{remark}

\begin{remark} One can also consider unordered simplices and their corresponding unordered distances  where one further identifies configurations which are permutations of each other. The counts change by a $(k+1)!$, the size of the corresponding symmetric group, which is independent of $q$ and so similar results to the ordered case are easily obtained.\end{remark} 

\begin{definition} Let $Q$ be a non-degenerate quadratic form on $\mathbb{F}_q^d$. We say that two $k$-simplices in $\mathbb{F}_q^d$ with vertices $(x_1, . . ., x_{k+1}), (y_1, . . ., y_{k+1})$ are in the same similarity class if there exists $\theta \in O(Q)$, $z \in \mathbb{F}_q^d$, and $r \in \mathbb{F}_q^*$ so that $z + r\theta(x_i) = y_i$ for $i = 1, . . . , k+1$. Equivalently, two (non-degenerate) simplices are in the same $Q$-similarity class if
there exists a dilation $r \in \mathbb{F}_q^*$ so that $\|x_i-x_j\| = \|r(y_{i}-y_{j})\|$ for $1\leq i < j \leq k$. We can then define each similarity class by
a distance matrix $\mathbb{D}$ as before, modulo nonzero scaling.
We denote the set of similarity classes of $k$-simplices determined by $E \subset {\mathbb F}_q^d$ by $S^d_{k,Q}(E)$ and
by $S^d_k(E)$ in the case of the standard dot product. \end{definition}

In terms of the paradigm described above, we have that $ \Lambda(x^1, \dots, x^{k+1})$ is determined by ${k+1 \choose 2}$ entries $||x^i-x^j||$, $i < j$. Therefore, the question we ask is, ``How large does $E \subset {\mathbb F}_q^d$ need to be to ensure that $|T^d_k(E)|\gg q^{k+1 \choose 2}$ and $|S^d_k(E)|\gg q^{{k+1 \choose 2}-1}$?" Here and throughout, $X\ll Y$ means that there exists $C>0$ such that $X \leq CY$, and $|S|$ (with $S$ a finite set) denotes the number of elements of $S$. Substantial work has been done in this direction over the past few years. See, for example, \cite{CEHIK12}, \cite{BIP12} and the references contained therein. In this paper we improve the previous exponents considerably and, more importantly, we introduce a group theoretic perspective on this question which, in addition to contributing to the problem at hand, opens up a number of fruitful research directions.

The basic starting point behind our approach in this paper is a simple observation that if $d=2$ and
$$||x-y||=||x'-y'|| \not=0,$$ then there exists a unique $\theta \in SO_2({\mathbb F}_q)$ (the group of orthogonal matrices with determinant 1) such that $x'-y'=g(x-y)$. When $d>2$, $\theta$ is not unique, but we can achieve uniqueness by considering the quotient $O_d({\mathbb F}_q)/ O_{d-1}({\mathbb F}_q)$, i.e., there
is a unique $O_{d-1}({\mathbb F}_q)$-coset of elements $g$ with the property. As a result the study of the distance set problem in ${\mathbb F}_q$ reduces to the investigation the the $L^2$ norm, in ${\mathbb F}_q^d$ and $O_d({\mathbb F}_q)$, of the counting function
$$ \nu_{\theta}(x)=|\{(u,v) \in E \times E: u-\theta v=x \}| \ \text{with} \ \theta \in O_d({\mathbb F}_q).$$

In the case of $k$-simplices, the group theoretic approach leads to the study of the $L^{k+1}$ norm of $\nu_g$, which proves considerably more efficient than the bootstrapping approach in \cite{CEHIK12} or the combinatorial techniques in \cite{BIP12}. Moreover, this approach gives us results in the most natural case $k=d>2$ which was outside of the reach of the previously employed methods in dimensions three and greater.

\vskip.125in

\subsection{Statement of main results}

Our first result gives us a non-trivial exponent for any class of simplices in ${\mathbb F}_q^d$ for any non-degenerate quadratic form $Q$.

\begin{theorem} \label{main}  Let $Q$ be a non-degenerate quadratic form on $\mathbb{F}_q^d$ where $q$ is odd and $d \geq 2$. Let $E \subset \mathbb{F}_q^d$.
There exist constants $C$, $c$, depending only on $1 \leq k \leq d$, such that:

if ${\displaystyle |E| \geq  C q^{d - \frac{d-1}{k+1}}}$ then ${\displaystyle |T^d_{k,Q}(E)| \geq  cq^{{k+1} \choose 2}}$.

\end{theorem}

\begin{theorem} \label{main2}
Let $E \subset \mathbb{F}_q^2$, where $q$ is odd. There exist constants $C$, $c$, such that:

\begin{equation}\label{triang}\mbox{if } |E| \geq Cq^{\frac{8}{5}}, \mbox{ then } |T^2_2(E)|\geq  c{q^3};\end{equation}
\begin{equation}\label{dist}\mbox{if } |E| \geq q^{\frac{4}{3}}, \mbox{ then } |T^2_1(E)|\geq  c{q}.\end{equation}
\end{theorem}

\begin{remark} Theorem \ref{main} improves over the best currently-known result of $q^\frac{d+k}{2}$ from \cite{CEHIK12} as long as the condition $k > d-2$ holds. It is important to note that the case $k=d$ is perhaps the most natural case as it deals with $d$-dimensional simplices in $d$-dimensional space.

The estimate \eqref{triang} of Theorem \ref{main2} improves over the result of $|E| \gg  q^{\frac{7}{4}}$ obtained in \cite{BIP12}. The estimate \eqref{dist} extends the  result  in \cite{CEHIK12} (Theorem 2.2), valid for $q\equiv 3\mbox{ mod }4$, to the case $q\equiv 1\mbox{ mod }4$.

The values of the constants $C,c$ can be easily deduced from the proofs below where they often appear explicitly.\end{remark}

The estimate \eqref{dist} of Theorem \ref{main2} (and its proof) has the following sum-product inequality as a corollary. Let 
$$ A\cdot B=\{ab:\,a\in A, b\in B\}. $$

\vskip.125in 

\begin{corollary}\label{spcor} Let $q$ be an odd prime and $X,Y\subseteq \mathbb F_q.$ There exist constants $C$, $c$, such that

if ${\displaystyle |X||Y| \geq Cq^{\frac{4}{3}}}$ then 
$$ |(X\pm X)\cdot (Y\pm Y)| \geq cq.$$ 
\end{corollary}

\vskip.125in

Our next result deals with the similarity classes of triangles in dimension two.

\begin{theorem} \label{similar} Let $E \subset \mathbb{F}_q^2$, with $q$ odd. There exist constants $C$, $c$, such that:

if  ${\displaystyle |E| \geq  Cq^{\frac{4}{3}}},$  then  ${\displaystyle |S^2_2(E)| \ge c q^2.}$ \end{theorem}

\vskip.125in

We also obtain non-trivial exponents for the distribution of simplices on $d$-dimensional spheres. In principle, Theorem \ref{main} can be used directly in this case, where we replace $d$ with $d+1$ since $d$-dimensional spheres are subsets of ${\mathbb F}_q^{d+1}$. Unfortunately, in the case of $d$-dimensional simplices, Theorem \ref{main} yields the exponent $d+1-\frac{d}{d+1}>d$. Since we are considering simplices that lie on a $d$-dimensional sphere, we are looking for exponents $<d$ since the size of a $d$-dimensional sphere is $q^d(1+o(1))$. In order to do this, the geometry of the sphere must be taken into account in a careful way and we are able to accomplish this in the following result.

\begin{theorem} \label{sphere} Let $Q$ be a non-degenerate quadratic form. Let $E \subset S_Q^d=\{x \in {\mathbb F}_q^{d+1}: Q(x)=r \}$, the $d$-dimensional sphere of radius $r \neq 0$ centered at the origin in ${\mathbb F}_q^{d+1}$. Let $2 \leq k \leq d$ and suppose that $|E| \gg  q^{d - \frac{d-1}{2(k-1)}}$ when $k \geq 3$ and
$|E| \gg q^{\frac{2d+1}{3}}$ when $k=2$. Then $|T^d_k(E)|\gg  q^{{k+1} \choose 2}$. \end{theorem}

\vskip.125in 

The case of $k = 1$ is treated in \cite{HIKR11}.

\vskip.125in

\subsection{Sharpness of results}

Fix a non-degenerate quadratic form $Q$ on $\mathbb{F}_q^d$.
Let $\alpha_{k,d}(Q)$ denote the sharp exponent for the $T^d_{k,Q}(E)$ problem in ${\mathbb F}_q^d$. More precisely, fix $k \leq d$ and let $\alpha_{k,d}(Q)$ be the infimum of numbers $s>0$ such that if $|E|\gg q^s$, then $|T^d_{k,Q}(E)|\gg q^{k+1 \choose 2}$. In the case of the standard dot product, we will
denote $\alpha_{k,d}(Q)$ by $\alpha_{k,d}$. We also say that the exponent $\alpha_{k,d}(Q)$ is {\em attainable up to the endpoint} if for every $|E|\gg q^s$, $|T^d_{k,Q}(E)|\gg q^{k+1 \choose 2}$, with a universal choice of constants implicit in the $\gg$ notation.

\begin{theorem} \label{sharpness} We have the following lower bounds on the sharp exponent for the simplex problem when $1 \le k \le d$:

\vskip.125in

i) $\alpha_{1,d} \ge \frac{d+1}{2}$ if $d \ge 3$ is odd, attainable up to the endpoint.

\vskip.125in

ii) $\alpha_{1,d} \ge \frac{d}{2}$ if $d \ge 2$ is even, unattainable up to the endpoint.

\vskip.125in

iii) $\alpha_{k,d}(Q) \ge k-1+1/k$ for any non-degenerate quadratic form $Q$.

\end{theorem}

\vskip.125in
\begin{remark} We have included (ii) in Theorem \ref{sharpness} for the following purpose. It appears that there is a genuine geometric reason why (i) cannot hold in even $d\geq 4$. We have not been able to identify it with any rigor. However, we can provide counterexamples, showing the presence of at least an extra logarithmic term for the minimum number of distinct  distances in {\em all} even dimensions, in contrast to the Euclidean case, where an equivalent of a logarithmic term only appears in the two-dimensional case. \end{remark}

\vskip.25in 

\section{Proof of the main result}

We'll begin by providing a bound on the $\ell^n$ norm of a real-valued function over a finite space. This essentially allows us to reduce $k$-simplices to the case of $2$-simplices.

\begin{lemma} \label{L1}
For any finite space $F$, any function $f: F \rightarrow \bb{R_{\geq 0}}$, and any $n \geq 2$ we have
$$\sum_{z \in F} f^n(z) \leq |F|\left( \frac{\|f\|_1}{|F|} \right)^n + \frac{n(n-1)}{2}\|f\|_\infty^{n-2}\sum_{z \in F} \left(f(z) - \frac{\|f\|_1}{|F|} \right)^2,$$
where $\|f\|_1=\sum_{z \in F} |f(z)|,$ and $\|f\|_\infty=\max_{z \in F} f(z)$.
\end{lemma}

\begin{proof}

By Taylor's formula with remainder, we have
$$y^n = y_0^n + ny_0^{n-1} (y-y_0) + \frac{n(n-1)}{2} c^{n-2} (y-y_0)^2$$
where $c$ lies between $y$ and $y_0$.
Plugging in $y=f(z)$ we get:
$$
f(z)^n = y_0^n + ny_0^{n-1}(f(z)-y_0) + \frac{n(n-1)}{2} c^{n-2} (f(z)-y_0)^2
$$
where $c$ lies between $y_0$ and $f(z)$. If $y_0$ is chosen to be in $[0, \|f\|_{\infty}]$, then $c^{n-2} \leq \|f\|_{\infty}^{n-2}$, so we obtain
$$
f(z)^n \leq y_0^n + ny_0^{n-1}(f(z)-y_0) + \frac{n(n-1)}{2} \|f\|_{\infty}^{n-2} (f(z)-y_0)^2
$$

%By the binomial theorem,  we have

%$$\begin{aligned} f^n(z) &= \sum_{i=0}^n {n \choose i}A^{n-i} (f(z) - A)^i\\ & = (1-n)A^n + %nA^{n-1}f(z) + \sum_{i=2}^n  {n \choose i}A^{n-i} (f(z) - A)^i
%\\& = (1-n)A^n + nA^{n-1}f(z) +  (f(z) - A)^2\sum_{i=2}^n {n \choose i}A^{n-i} (f(z) - A)^{i-2}.%\end{aligned}$$

%Now notice that both $A$ and $f(z)-A$ are less than $\|f\|_\infty$, so we replace them to get

%$$f^n(z) \leq (1-n)A^n + nA^{n-1}f(z) +  (f(z) - A)^2\sum_{i=2}^n {n \choose i} \|f\|%_\infty^{n-2}$$

%Now we recall that $\sum_{i=0}^n {n \choose i} = 2^n$, so

%$$f^n(z) \leq A^n + nA^{n-1}(f(z)-A) +  2^n\|f\|_\infty^{n-2}(f(z) - A)^2.$$

Plugging in $\frac{\|f\|_1}{|F|}$ for $y_0$ and summing over all elements of $F$ on both sides gives the result.

\end{proof}

Let $V$ be the $\mathbb{F}_q$-vector space of $(k+1)\times(k+1)$ symmetric matrices
which will function as the space of possible ordered $k$-simplex distances.
Let $E \subset \mathbb{F}_q^d$ and define $\mu: V \rightarrow \mathbb{Z}$ by the relation

$$\mu(\mathbb{D}) = \# \{ (x_1,. . .,x_{k+1}) \in E^k : \|x_i-x_j\| = d_{i,j}, 1 \leq i <j \leq k+1 \}.$$

\vskip.125in

Notice that the number of distinct congruence classes represented by $(k+1)$-tuples in $E$ is then

$$ T^d_k(E) =\sum_{\mathbb{D} \in supp(\mu) } 1$$

A warning to the reader: in reality we should parametrize the above sum by the abstract set of congruence classes of $k$-simplices in $\mathbb{F}_q^d$. However we choose in the
description to use the ordered distances $\mathbb{D}$ instead. For non-degenerate simplices there is no difference, and in fact during our proofs
we use the group congruences exclusively and not the distances. This ``simpler" description is made only for brevity and effects results in no significant way.

Another abuse of notation we make deals with the size of spheres and notation for stabilizers. In general there are three distinct sizes of spheres in
$\mathbb{F}_q^d$, depending on whether the radius is zero, a nonzero square, or a non-square (see the appendix for details). This means that there are
three types of stabilizers of nonzero elements by $O_d(\mathbb{F}_q)$ or $O(Q)$ in general. By abuse of notation, we will denote them all by $O_{d-1}(\mathbb{F}_q)$ (the stabilizer of an element of norm 1)
though in reality these stabilizers can have 3 different sizes. However since the sizes of spheres (in nearly all cases) are all of the same order of magnitude, so are the sizes of these stabilizers and this abuse of notation is reasonable and avoids the clutter of overzealous bookkeeping. It does mean, however, that each time we write
$|O_{d-1}(\mathbb{F}_q)|$ as the size of some stabilizer group, we implicitly mean to include a $(1+o(1))$ factor. In all of our arguments but the proof of Theorem~\ref{sphere} the only effect of these subtleties is a change in the final implicit constants in the statement of the theorem over those obtained in the proof.

\vskip.125in

\subsection{Proof of Theorem \ref{main}}
Fix a non-degenerate quadratic form $Q$ on $\mathbb{F}_q^d$ and let $1 \leq k \leq d$. All norms, distances, congruences and inner products in this section will be with respect to $Q$. Without loss of generality assume $\vec{0}\not\in E$.

\vskip.125in 

By Cauchy-Schwarz, we have

\begin{equation}\label{CS}T^d_{k,Q}(E) \geq \frac{\left( \sum_{\mathbb{D}} \mu(\mathbb{D}) \right) ^2}{\sum_{\mathbb{D}} \mu^2(\mathbb{D})}.\end{equation}

Observe that the numerator simply counts the total number of $k$-simplices in our set $E$, and then squares it. Since each simplex is determined by a $(k+1)$-tuple of points, our numerator is $|E|^{2k+2}$.

The denominator, on the other hand, counts the number of pairs of congruent $k$-simplices. It follows that

$$ \sum_{\mathbb{D}} \mu^2(\mathbb{D}) =$$

$$ \# \{(x_1, . . ., x_{k+1},y_1, . . ., y_{k+1}) \in E^{2k+2} :\; \exists \ \theta \in O(Q), z \in \mathbb{F}_q^d:\; \theta(x_i) + z = y_i, \,i=1,\ldots,k+1\}.$$

\vskip.125in

Let $\nu(\theta,z) = \# \{(u,v) \in E \times E: u-\theta (v) = z \}$. The quantity $\nu(\theta,z)$ equals the number of pairs of points $(u,v)\in E\times E$ such that the ``rigid motion'' $\rho (z,\theta)$, which is a composition of $\theta \in O(Q)$ and a translation by $z$, acts as $u = \rho(z,\theta)\, v$. Since the variables $\theta,z$ are treated rather differently, we will also be using the notation $\nu(\theta,z)=\nu_\theta(z)$

Then $\nu^{k+1}_\theta(z)$ equals the number of $(k+1)$-tuples of such pairs of points $(u,v)$. This is equal to the number of pairs of congruent $k$-simplices $(x_1,\dots,x_{k+1}), (y_1,\dots,y_{k+1})$, with vertices in $E$, mapped into one another by the transformation $\rho (z,\theta)$. This is because such $k$-simplices determine and are determined by $(k+1)$ pairs of points $(x_i,y_i), 1 \leq i \leq k+1$ with $\rho(z,\theta)x_i=y_i$.

The transformation $\rho (z,\theta)$ taking a $(k+1)$-tuple of elements of $E$ to another is defined up to the stabilizer of the first $(k+1)$-tuple. Thus this pair of congruent
$k$-simplices will contribute $+1$ toward the count in $\nu^{k+1}_\theta(z)$ for a number
of $(z,\theta)$ pairs equal to the size of this stabilizer.

Congruent $k$-simplices will have conjugate stabilizers, thus we may define
$s(\mathbb{D})$ to be the common stabilizer size of $k$-simplices in the congruence class
$\mathbb{D}$. From the above arguments it follows that,

\begin{equation}
\sum_{\mathbb{D}} s(\mathbb{D}) \mu^2(\mathbb{D}) = \sum_{\theta,z} \nu^{k+1}_{\theta}(z)
\end{equation}

Therefore, we can write

\begin{equation}\sum_{\mathbb{D}} \mu^2(\mathbb{D}) \leq \frac{1}{|O_{d-k}(\mathbb F_q)|} \sum_{\theta,z} \nu_\theta^{k+1}(z),\label{jack}\end{equation}
where $|O_{d-k}(\mathbb F_q)|$ is the {\bf minimum} size of the stabilizer of a $(k+1)$-tuple, that is, of a $k$-dimensional subspace of $\mathbb F_q^d$. (The more degenerate the $(k+1)$-tuple, the bigger the stabilizer. Here we use $|O_{d-k}(\mathbb{F}_q)|$ as all orthogonal groups of the same dimension
have comparable sizes, independent of $Q$. See the appendix for details. )

\vskip.125in

Applying lemma \ref{L1}, we have, for each $\theta$:

\begin{equation}\label{binder}\begin{aligned} \sum_{z} \nu_\theta^{k+1}(z) & \leq q^d\left( \frac{\|\nu_\theta\|_1}{q^d} \right)^{k+1} + \frac{(k+1)k}{2}\|\nu_\theta\|_\infty^{k-1}\sum_{z} \left(\nu(\theta,z) - \frac{\|\nu_\theta\|_1}{q^d} \right)^2\\
& = q^d\left( \frac{\|\nu_\theta\|_1}{q^d} \right)^{k+1} + \frac{(k+1)k}{2}q^d\|\nu_\theta\|_\infty^{k-1} \sum_{\xi \neq 0} |\hat{\nu}_\theta(\xi)|^2\\
&\leq
q^{-kd}|E|^{2k+2} + \frac{(k+1)k}{2}q^d|E|^{k-1} \sum_{\xi \neq 0} |\hat{\nu}_\theta(\xi)|^2,
\end{aligned}\end{equation}
using the fact that
$\displaystyle \sum_{ z}  \nu(\theta,z) = |E|^2$ and $\nu(\theta,z) \leq |E|$.

Also observe that $$\begin{aligned}\nu_\theta (z) = \sum_v E(v)E(z + \theta v) & = \sum_{v,\alpha} E(v)\hat{E}(\alpha)\chi(\alpha \cdot (z + \theta v))  \\ & = q^d\sum_{\alpha} \hat{E}(\alpha)\hat{E}(-\theta^{T} (\alpha)) \chi(z \cdot \alpha ).\end{aligned}$$

It follows that $$\hat{\nu}_\theta(\xi)  = q^d \hat{E}(-\xi)\hat{E}(\theta^{T} (\xi)),$$

and thus

\begin{equation}\sum_{\theta, \xi \neq 0} |\hat{\nu}_\theta(\xi)|^2 = q^{2d}\sum_{\theta, \xi \neq 0}|\hat{E}(\xi)|^2|\hat{E}(\theta^{T}(\xi))|^2.\label{fform}\end{equation}

Furthermore, notice that given $\xi\neq 0$, the action $\theta^{T}(\xi)$ is defined up to the stabilizer of $\xi$, which has size at most $|O_{d-1}(\mathbb F_q)|$ when $d \geq 2$ (It is actually of equal order except in the case $d=2$ when circles of radius $0$ arise). We conclude that

\begin{equation}\sum_{\theta, \xi \neq 0} |\hat{\nu}_\theta(\xi)|^2 \leq q^{2d}|O_{d-1}(\mathbb F_q)|\sum_{ \xi \neq 0}|\hat{E}(\xi)|^2\sum_{\substack{\eta\neq 0, \\ \| \eta \| = \|\xi\|}}|\hat{E}(\eta)|^2.\label{matt}\end{equation}

Extending the summation in $\eta$ over all $\eta\neq 0$ and using Plancherel twice, we get

$$\sum_{\theta, \xi \neq 0} |\hat{\nu}_\theta(\xi)|^2 \leq |O_{d-1}(\mathbb F_q)||E|^2.$$

We plug this back into (\ref{binder}) and get

$$ \begin{aligned} \sum_{\mathbb{D}} \mu^2(\mathbb{D})
& \leq  \frac{1}{|O_{d-k}(\mathbb F_q)|} \left(   |O(Q)|\frac{|E|^{2k+2}}{q^{kd}} + \frac{(k+1)k}{2}q^d|E|^{k+1}|O_{d-1}(\mathbb F_q)|\right)
\\ &= \frac{C_1|E|^{2k+2} +  C_2|E|^{k+1}q^{kd+1}}{q^{{k+1} \choose 2}},\end{aligned}$$
using $|O_{n}(\mathbb F_q)| = C_n q^{\frac{n(n-1)}{2}}=|O(Q)|(1+o(1))$. The constants $C_1$ and $C_2$ depend only on $d$ and $k$ and
the quadratic form $Q$.

This, in view of (\ref{CS}) gives us

$$T^d_{k,Q}(E) \geq \frac{q^{{k+1} \choose 2}|E|^{2k+2}}{ C_1|E|^{2k+2} + C_2|E|^{k+1}q^{kd+1}}.$$

Suppose the first term in the denominator exceeds the second term. This happens exactly
when

$$|E|^{k+1} > \frac{C_2}{C_1} q^{kd+1}$$

and in this case it follows that

$$T^d_{k,Q}(E) \geq \frac{q^{{k+1} \choose 2}}{ 2C_1}.$$

This, after redefining the constants to fit the statement of of Theorem \ref{main}, completes its proof.

\vskip.125in

\subsection{Proof of Theorem \ref{main2}}

%If $d=2$ we can consider $\theta\in SO_2(\mathbb F_q)$, that is orthogonal matrices with determinant $1$.
We will refine the above proof of Theorem \ref{main} in the case $d=k=2$.
Using $d=k=2$ and $|O_2(\mathbb{F}_q)|=q$ in (\ref{binder}), it follows that

\begin{equation}\label{2dbd}\sum_{\theta, z} \nu_\theta^{3}(z) \leq q^{-3}|E|^6 + 3q^2|E|\sum_{\theta, \xi \neq 0} |\hat{\nu}_\theta(\xi)|^2.\end{equation}

We shall first address the technically slightly-more-straightforward case $q \equiv 3\;\rm{mod}\;4.$ If this is the case, the only $\xi\in \mathbb F_q^2$ with $\|\xi\|=0$ is the origin. Hence by (\ref{matt}):

\begin{equation}\label{secmom}\begin{aligned} \sum_{\theta, \xi \neq 0} |\hat{\nu}_\theta(\xi)|^2 &= q^{4}\sum_{t \in \mathbb{F}_q} \sum_{\substack{\xi, \eta \neq 0 \\ \| \xi \| = \|\eta\| = t}} |\hat{E}(\xi)|^2|\hat{E}(\eta)|^2\\
&= q^4\sum_{t \in \mathbb{F}_q^*} \left(\sum_{ \| \xi \| = t} |\hat{E}(\xi)|^2 \right)^2 = q^4\sum_{t \in \mathbb{F}_q^*}\sigma^2_E(t) =q^4 M_E(q),\end{aligned}\end{equation}

in the notation of the following lemma, proven  in \cite{CEHIK12} (Lemma 4.4).

\begin{lemma}\label{mlem}
Let $E \subset \mathbb{F}_q^2$. Define
$$ \begin{array}{ccc}\sigma_E(t) & = &\sum_{\|\xi\| = t} |\hat{E}(\xi)|^2, \\ \hfill \\ M_E(q) &=& \sum_{t \in \mathbb{F}_q^*} \sigma_E^2(t).\end{array}$$
\indent One has the estimates \begin{equation}\label{spav}\forall t\neq 0,\;\sigma_E(t)\leq \frac{\sqrt{3}|E|^{\frac{3}{2}}}{q^3},\qquad\displaystyle M_E(q) \leq \frac{\sqrt3|E|^{5/2}}{q^5}.\end{equation}
\end{lemma}

Applying Lemma \ref{mlem} yields immediately

\begin{equation}\label{normal}\sum_{\theta, \xi \neq 0} |\hat{\nu}_\theta(\xi)|^2 \leq  \frac{\sqrt3 |E|^{5/2}}{q}.\end{equation}

Putting this together with (\ref{2dbd}) and (\ref{CS}), we have

\begin{equation}T^2_2(E) \geq \frac{|E|^6}{(q^{-3}|E|^6 + 3\sqrt3q|E|^{7/2})}.\label{t2}\end{equation}

Hence, $T^2_2(E)\geq \frac{q^3}{2}$,  whenever $|E| \geq (3\sqrt{3})^{\frac{2}{5}} q^{\frac{8}{5}}.$

\medskip
Let us now deal with the case $q \equiv 1\;\rm{mod}\;4$, where the origin is not the only $\xi\in \mathbb F_q^2$ with $\|\xi\|=0$. Let $\iota^2=-1$, and consider the vectors $n_\pm = (1,\pm \iota)\in \mathbb F_q^2$,
respectively, in the standard basis. These are null vectors, that is $\|n_\pm\|=0$, and they span one-dimensional subspaces $L_{\pm}$, respectively. Observe that the Fourier transform of
each of these subspaces is supported on the subspace itself. Indeed, let $x=(x_+,x_-)$ be coordinates of a vector in $\mathbb F_q^2$ with respect to the basis $\{n_+,n_-\}$, $\xi=(\xi_+,\xi_-)$ being the  coordinates in the dual space, with respect to the same basis. Then the characteristic function of $L_+$ is  given by $L_+(x_+,x_-) = \delta(x_-),$ that is $1$ if $x_-=0$ and $0$ otherwise. For the Fourier transform of $L_+$ we then have
\begin{equation}\label{nsft}\hat{L}_{+}(\xi_+,\xi_-) = \frac{1}{q^2}\sum_{x_+} \chi(-(\xi_+n_+ + \xi_- n_-)\cdot x_+ n_+) = \frac{1}{q} \delta(\xi_-).\end{equation}

The technical problem is that the estimate (\ref{spav}) for the quantity $\sigma_E(t)$ does not apply to the case $t=0$. Indeed, let $E$ be a union of $r$ parallel translates of $L_+$. Denote $S_0=L_+\cup L_-$, the ``null circle''.

Then \begin{equation}\label{clct}\begin{aligned}\sum_{\xi \in L_+} |\hat{E}(\xi)|^2 = \sum_{\xi} |\hat{E}(\xi)|^2 L_+(\xi)  &= \frac{1}{q^2}\sum_{x} \left(\frac{1}{q^2}\sum_y E(y)E(x+y)\right)(qL_+(x))\\
& = \frac{1}{q^3}\#\{(u,v)\in E\times E: u-v\in L_+\}=\frac{rq}{q^2}=\frac{|E|}{q^2},  \end{aligned}\end{equation}
which is worse than the first estimate in (\ref{spav}) for $|E|\ll q^2$.

To rule out this example we proceed as follows. Given $E$, we call the $x_\pm$ coordinate ``rich" (resp.``poor"), if there are at least (resp. fewer than) $2\sqrt{|E|}$ points of $E$ with this coordinate. We will also we call the $x_\pm$ coordinate ``wealthy" (resp. ``impoverished") if it contains at least (resp. fewer than) $\sqrt{2|E|}$ points of $E$. Suppose that there are $m$ wealthy $x_+$ coordinates and $n$ wealthy $x_-$ coordinates. The number of points that are wealthy in both coordinates must thus be less than $mn$. Notice that $m(\sqrt{2|E|})$ and $n(\sqrt{2|E|})$ must both be $\leq |E|$, and thus $mn \leq \frac{|E|}{2}$. We will now discard any points of $E$ for which both coordinates are wealthy, which eliminates at most half of the points of $E$. Call this new set $E'$. As a result, for any $(x_+,x_-) \in E'$, either $x_+$ or $x_-$ is impoverished with respect to the original set $E$. It follows that, for any $(x_+,x_-) \in E'$, either $x_+$ or $x_-$ is poor with respect to the new set $E'$, since $\sqrt{2|E|} = 2\sqrt{\frac{|E|}{2}} \leq 2\sqrt{|E'|}$. For simplicity, we will just refer to the new set $E'$ as $E$.

Now two options are left: a) both coordinates are poor, and b) one coordinate is rich and the other is poor.

\medskip
In the case a) we repeat the calculation (\ref{clct}) which yields:
\begin{equation}\begin{aligned} \sum_{\xi \in L_+\cup L_-} |\hat{E}(\xi)|^2 & \leq  \frac{1}{q^3}\#\{(u,v)\in E \times E: u-v\in L_+\cup L_-\}\\& \leq \frac{2}{q^3} (2\sqrt{|E|})^2 \sqrt{|E|} \;\leq \; \frac{8|E|^{\frac{3}{2}}}{q^3}.\end{aligned}\label{nullcount}\end{equation}
Thus the estimate (\ref{normal}) which worked in the case $q \equiv 3\mbox{ mod }4$ is modified as follows:
$$\sum_{\theta, \xi \neq 0} |\hat{\nu}_\theta(\xi)|^2 \leq  \frac{\sqrt3 |E|^{5/2}}{q} +\frac{ 64|E|^3}{q^2}.$$
Since we can always assume that $|E|<cq^{2}$ for a small enough $c$, the first term in this estimate dominates, and the same final conclusion as in the case $q \equiv 3 \mbox{ mod 4 }$ can be made.

\medskip In the case b), suppose that the $x_+$ coordinate is rich and $x_-$ is poor. Then partition $E = E_1\cup E_2$, where the sets $E_1$ and $E_2$ have approximately equal size and their projections on the $n_+$ direction are disjoint. (This is possible as long as $|E|>Cq$ for a large enough $C$.) We then consider a variant of the procedure, set up throughout the proof of Theorem \ref{main}, pertaining to the case $k=d=2$. In this situation, we only consider triangles which have at least one vertex in $E_1$ and at least one in $E_2$. This pruning resolves the technical problem of ``null'' distances in the case $q \equiv 1\mbox{ mod }4$. Indeed, we will now have the same estimate as (\ref{nullcount}) for the quantity
$$\frac{1}{q^3} \#\{(u,v)\in E_1\times E_2: u-v\in L_+\cup L_-\}.$$

While much of this procedure simply repeats what has been done so far, it also calls for some additional notation, so we relegate the specific details to the forthcoming subsection.

\medskip
This completes the proof of Theorem \ref{main2} \qed

\medskip
\subsubsection{Proof of Theorem \ref{main2}, case b)}
We repeat the formalism of Theorem \ref{main} in the case $d=2$ for the case of the sets $E_1,$ $E_2$.

For $i=1,2$, set $\nu_i(\theta,z) = \# \{(u,v) \in E_i \times E_i: u-\theta (v) = z \}$. The quantity $\nu_i(\theta,z)$ equals the number of pairs of points $(u,v)\in E_i\times E_i$ such that the ``rigid motion'' $\rho (z,\theta)$, which is a composition of $\theta \in SO_2(\mathbb F_q)$ and a translation by $z$, acts as $u = \rho(z,\theta)\, v$. Hence $\displaystyle \sum_{ z}  \nu_i(\theta,z) = |E_i|^2$ and $\nu_i(\theta,z) \leq |E_i|$.

Then $\nu_1\nu_2(\theta,z)$ equals the number of  pairs of congruent segments, with one vertex in $E_1$ and the other in $E_2$, so that one gets mapped to the other by the transformation $\rho(z,\theta)$.
The quantity $\nu_1\nu_2(\nu_1+\nu_2)(\theta,z)$ gives an upper bound on the number of pairs of congruent triangles, mapped into one another by the transformation $\rho(z,\theta),$ such that the three vertices of the triangle do not lie in $E_1$ or $E_2$ alone.

Hence, we need to estimate the quantity
\begin{equation}\label{oke}\begin{aligned}
\sum_{\theta,z} \nu_1\nu_2(\theta,z) & \leq \frac{q+1}{q^{2}} |E_1|^{2}|E_2|^{2} + 4q^2 \sum_{\xi\neq 0, \theta }\hat\nu_1(\theta, \xi)\overline{\hat\nu_2(\theta,\xi)} \\
& = \frac{q+1}{q^{2}} |E_1|^{2}|E_2|^{2} + 4q^6 \sum_{\xi\neq 0, \theta } \hat{E_1}(\xi)\overline{\hat{E_2}(\xi)}\overline{\hat{E_1}(\theta\xi)}\hat{E_2}(\theta\xi)\\
& = \frac{q+1}{q^{2}} |E_1|^{2}|E_2|^{2} + 4q^6 \sum_{t\in \mathbb F_q} \sigma_{E_1,E_2}(t)^2.\end{aligned}\end{equation}

where $S_t=\{\xi:\|\xi\|=t\}$ and ${\displaystyle \sigma_{E_1,E_2}(t) = \sum_{\xi\in S_t} \hat{E_1}(\xi)\overline{\hat{E_2}(\xi)}}.$

For $t\neq 0$ we apply Cauchy-Schwarz to get
${\displaystyle |\sigma_{E_1,E_2}(t)|\leq \sqrt{\sigma_{E_1}(t)\sigma_{E_2}(t)}}$, and using (\ref{spav}) conclude that

\begin{equation}\label{okest}
\sum_{t\in \mathbb F^*_q} \sigma_{E_1,E_2}(t)^2 \leq \frac{\sqrt{3}}{q^5} \left(|E_1|^{\frac{5}{2}}+ |E_2|^{\frac{5}{2}}\right)\leq \frac{\sqrt{3}|E|^{\frac{5}{2}}}{\sqrt{8}q^5}.
\end{equation}
In addition, by the construction of the sets $E_1$, $E_2$, we get an inequality similar to \eqref{nullcount}, namely
\begin{equation}\label{okest0}
\sigma_{E_1,E_2}(0) \leq \frac{1}{q^3}|\{u,v\in E_1\times E_2: u-v\in L_+\cup L_-\}|\leq (2\sqrt{|E|})^2 \sqrt{|E|} = \frac{4|E|^\frac{3}{2}}{q^3}.
\end{equation}
Indeed, for $(u,v)\in E_1\times E_2$, the difference $u-v$ cannot lie in $L_+$ while the $L_-$-coordinate is poor, that is the multiplicity of the difference cannot exceed $2\sqrt{|E|}$.

We conclude that as long as $|E|$ is sufficiently small relative to $q^2$ (to ensure the dominance of the estimate (\ref{okest}) over (\ref{okest0})) and
\begin{equation}\label{foth}
|E|^{\frac{3}{2}} \geq 32 q^2
\end{equation}
(to ensure the dominance of the first term in the right-hand side of (\ref{oke})), then
$$\sum_{\theta,z} \nu_1\nu_2(\theta,z) \leq 2\frac{1}{q} |E_1|^{2}|E_2|^{2}.
$$
It follows that
$$
T_1^2(E)\geq \frac{|E_1|^2|E_2|^2 } {\frac{2}{q}  |E_1|^2|E_2|^2 } = \frac{q}{2},
$$
which proves the claim of Theorem \ref{main2} about the quantity $T^2_1(E)$.

\medskip

We now turn to the quantity $T^2_2(E)$. Likewise in Lemma \ref{L1} we have
$$
\sum_{\theta,z} \nu_1^2\nu_2(\theta,z) = \frac{1}{q^4} \sum_\theta \|\nu_1(\theta)\|^2_1\|\nu_2(\theta)\|_1^2 + 8|E| \sum_{\theta,z} \tilde\nu_1(\theta,z)\tilde\nu_2(\theta,z),
$$
where for $i=1,2$
$$
\|\nu_i(\theta)\|_1 = \sum_{z}\nu_i(\theta,z),
\qquad  \tilde\nu_1(\theta,z) = \nu_i(\theta,z)  - \frac{1}{q^2}\|\nu_i (\theta) \|_1.$$

It follows (we now assume the reverse of (\ref{foth}) so that the second term dominates the right-hand side of (\ref{oke})) that
$$
\sum_{\theta,z} \nu_1\nu_2(\nu_1+\nu_2)(\theta,z) \leq  q^{-3}|E|_1^2|E_2|^2|E|^2 + 16q\frac{\sqrt{3}|E|^{\frac{7}{2}}}{\sqrt{8}q^5}.
$$
The first term in the latter estimate dominates if
$$
|E|^\frac{5}{2}\geq 64q^4,
$$
in which case, by Cauchy-Schwarz,
$$
T^2_2(E)\geq \frac{1}{12} q^3.
$$

\vskip.25in 

\section{Proof of Theorem \ref{similar}}
This proof proceeds similarly to that of Theorems \ref{main} and \ref{main2}, with some extra attention to the case $q \equiv 1\mbox{ mod }4$.

\begin{proof}

Without loss of generality assume $\vec{0} \not\in E$. Let $V_k$ once again be the vector space of symmetric $(k+1)\times(k+1)$ matrices which function as ordered distance vectors.
Let $\bar{V}_k$ denote this vector space modulo the equivalence relation of nonzero scaling
of matrices. We will denote the equivalence class of the distance matrix $\mathbb{D}$ in this
set by $\bar{\mathbb{D}}$. We define $\mu_S: V_2 \rightarrow \mathbb{Z}$ by

\begin{equation} \begin{array}{c} \mu_S(\mathbb{D}) = \\ \hfill \\
\# \{ (x_1,x_2,x_3) \in E^3 : \;\exists\ r \in \mathbb{F}_q^*, \|x_i-x_j\| = rd_{i,j}, 1 \leq i <j \leq 3 \}.\end{array}\end{equation}

As $\mu$ respects the equivalence relation of scaling, it induces a unique function
$\mu: \bar{V}_2 \to \mathbb{Z}$ via $\mu(\bar{\mathbb{D}})=\mu(\mathbb{D})$.

\vskip.125in

By Cauchy-Schwarz, we have

$$S^2_2(E) \geq \frac{\left( \sum_{\bar{\mathbb{D}}} \mu_S(\bar{\mathbb{D}}) \right) ^2}{\sum_{\bar{\mathbb{D}}} \mu_S^2(\bar{\mathbb{D}})}.$$

The numerator is $|E|^6$. The denominator counts the number of pairs of similar $2$-simplices. Clearly
\begin{equation} \begin{array}{c} \sum_{\bar{\mathbb{D}}} \mu_S^2(\bar{\mathbb{D}}) = \\ \hfill \\
 \# \{(x_1,x_2,x_3,y_1,y_2,y_3) \in E^6:\exists\, r \in \mathbb{F}_q^*,\theta \in O_{2}(\mathbb F_q),z \in \mathbb{F}_q^2:\, r\theta (x_i) + z = y_i, 1 \leq i \leq 3 \}.\end{array}\end{equation}
 
 \vskip.125in 

As in Theorem \ref{main2} we need to pay special attention to the case $q\equiv 1\mbox{ mod }4.$ Recall that $\iota^2=-1$, and consider the null vectors $n_\pm = (1,\pm \iota)\in \mathbb F_q^2$ (relative to the standard basis)  spanning one-dimensional subspaces $L_{\pm}$. The vectors $n_\pm$ are eigenvectors for any
$\theta\in SO_{2}(\mathbb F_q)$. In other words, the group $SO_{2}(\mathbb F_q)$ acts on $L_{\pm}$ as multiplication by $r\in \mathbb  F_q^*$.
Also recall that the Fourier transform of
each of these subspaces is supported on the subspace itself, see (\ref{nsft}).

We set, for $r\in\mathbb F_q^*$, $\theta\in SO_{2}(\mathbb F_q)$, and $z\in\mathbb F_q^2$,
$$\nu(r,\theta,z) = \# \{(u,v) \in E \times E: u-r\theta (v) = z \}. $$

Then

$$\sum_{\bar{\mathbb{D}}} \mu_S^2(\bar{\mathbb{D}}) = \sum^*_{r \neq 0,\theta,z} \nu^3(r,\theta,z),$$
where the meaning of $\sum^*_{r \neq 0,\theta,z}$ is as follows: Whenever $z\not \in (L_+ \cup L_-)\setminus\{0\}$, the summation is in $(r,\theta,z)$. Otherwise, it is done only in $(r,z)$.
By doing this, we avoid having to make stabilizer corrections, as with this convention
each $(u,v,z)$ choice determines unique values of $r$ and $\theta$.

We still have that  $\displaystyle \sum_{ z}  \nu(r,\theta,z) = |E|^2$ and $\nu(r,\theta,z) \leq |E|$. We also denote $\hat{\nu}(r,\theta,\xi)$ the
Fourier transform of $\nu$ with respect to the $z$ variable only.

Then, employing Lemma \ref{L1} again, we have

\begin{equation}\label{here}\sum^*_{r \neq 0,\theta,z} \nu^3(r,\theta,z) \leq q^2\sum_{r \neq 0,\theta}\left( \frac{|E|^2}{q^2} \right)^3 + 3q^2 |E| \sum^*_{r \neq 0,\theta,\xi \neq 0}
|\hat{\nu}(r,\theta,\xi)|^2.\end{equation}
The fact that the last sum also has a star on it is due to the fact that the Fourier transform of
the null subspace $L_\pm$ is supported on the subspace itself.

We also have, similar to (\ref{fform}):
$$\sum^*_{r \neq 0, \theta, \xi \neq 0} |\hat{\nu}(r,\theta,\xi)|^2 = q^4\sum^*_{r \neq 0, \theta, \xi \neq 0}|\hat{E}(\xi)|^2|\hat{E}(r\theta^{-1}\xi)|^2.$$

We now treat the case $\xi\in L_+ \cup L_-$ separately, when we have
\begin{equation}\label{nulls}
\sum^*_{r \neq 0, \theta,\, \xi \in (L_+ \cup L_-)\setminus\{0\}} |\hat{\nu}(r,\theta,\xi)|^2 = q^4 \sum_{r,s\neq 0,\pm} |\hat{E}(s n_{\pm})|^2|\hat{E}(s n_{\pm})|^2\\
\leq |E|^2.
\end{equation}

\vskip.125in

If $\xi\not\in L_+ \cup L_-$ (when $q\equiv3\mbox{ mod }4$ we simply require $\xi \neq 0$) and $\|\xi\|$ is a nonzero square (respectively non-square) in $\mathbb{F}^*_q$, then as we sum over all $\theta$ and $r$,
$r\theta^{-1}(\xi)$ varies over all vectors with square (respectively non-square) length. Each of these vectors gets represented two times, up to the sign. Thus, denoting $\lambda(\|\xi\|)=\pm 1$, as to whether
$\|\xi\|$ is a square or non-square in $\mathbb F_q^*$, we have:

\begin{equation}
\sum_{r \neq 0, \theta, \xi \not \in (L_+ \cup L_-)} |\hat{\nu}(r,\theta,\xi)|^2 =
2q^4 \sum_{\substack{\xi, \eta: \lambda(\| \xi \|)=\lambda(\|\eta\|)}} |\hat{E}(\xi)|^2|\hat{E}(\eta)|^2 \leq 2|E|^2.
\label{nnulls}\end{equation}

We plug the estimates (\ref{nulls}) and (\ref{nnulls}) back in (\ref{here}) and get, not having to increase the factor 2 in (\ref{nnulls}):

$$ \begin{aligned}\sum_{\bar{\mathbb{D}}} \mu_S^2(\bar{\mathbb{D}}) & \leq  \left(  (q-1)|O_{2}(\mathbb F_q)| \frac{|E|^6}{q^4} + 6q^2|E|^3 \right)\\
&= (1 + o(1))\left(\frac{|E|^6}{q^2} + 6q^2|E|^3\right).\end{aligned}$$

Suppose the first term in the last bracket dominates, that is $|E|\geq 2q^{\frac{4}{3}}$.
It follows that

$$S^2_2(E) \geq \frac{|E|^6}{2q^{-2}|E|^6} =\frac{q^2}{2}.$$

This completes the proof of Theorem \ref{similar}.
\end{proof}

\section{Proof of Theorem \ref{sphere}}

\begin{proof}

Fix $2 \leq k \leq d$.
Let $Q$ be a non-degenerate quadratic form on $\mathbb{F}_q^{d+1}$ and let $S^d$ denote a ``$d$-dimensional sphere'' of radius $r \neq 0$ in
$(\mathbb{F}_q^{d+1},Q)$ centered at the origin. Several times throughout this proof we will be using the fact that $|S^d| = q^d(1 + o(1))$ which is stated in Theorem~\ref{thm:spheresize} in the appendix (as long as $d \geq  1$) and in one case we will need the more explicit bound

$$|S^d| = q^d + O(q^{\frac{d}{2}})$$ which can also be found in the appendix. Let $E$ be a subset of $S^d$ and let $\tilde{T}_k^d(E)$ denote the $O(Q)$-congruence
classes of $k$-simplices determined by $E$. (This is the same as the number of $O(Q)$-congruence classes of pinned $k+1$-simplices determined by $E$,
where the first coordinate is pinned at the origin and is the only coordinate allowed to not lie in $E$.) It is clear that $|\tilde{T}_k^d(E)|=O(q^{{k+1 \choose 2}})$
and we will also use that $|O(Q)|=|O_{d+1}(\mathbb{F}_q)|(1+o(1))$ a fact proven in proposition~\ref{prop:orthogonalsize} in the appendix.

 We begin again with

\begin{equation}\label{sphereIneq}
\tilde{T}^d_k(E) \geq \frac{\left( \sum_{\mathbb{D}} \mu(\mathbb{D}) \right) ^2}{\sum_{\mathbb{D}} \mu^2(\mathbb{D})}
\end{equation}

as we did in the proof of the main result. The numerator, once again, is equal to $|E|^{2k+2}$. Next, we point out that for any pair $(K_1,K_2)$ of congruent $k$-simplices lying on $S^d$, the number of elements of $O(Q)$ rotating $K_1$ onto $K_2$ is the same
as the number of elements of $O(Q)$ that stabilize $K_1$. If the elements of
$K_1$ span a $m$-dimensional subspace, then $m \leq k+1$ and the stabilizer is of type
$O_{(d+1)-m}(\mathbb{F}_q)$ and is of order $q^{\binom{d+1-m}{2}}(1+o(1))$. The minimal
possible stabilizer size occurs in the non-degenerate case where the $k$-simplex in
$S^d$ spans a $(k+1)$ subspace and is of order $q^{\binom{d-k}{2}}(1+o(1))$.

%$$\frac{|O_{d+1}(\bb{F}_q)|}{|S^d|\prod_{i=1}^k q^{d-i}(1 + o(1))}.$$

%To derive this, we first count the number of distinct rigid motions on the sphere, which is %precisely the size of the orthogonal group. Then we count the number of distinct copies of %$K_1$ that can live on the sphere. There are $|S^d|$ choices for the first point of $K_1$. The %second point lies at a distance $\ell$ from the first point, so the number of choices for the %second point is the size of the intersection of $S^d$ and the sphere of radius $\ell$ centered at %the first point. This intersection is a sphere of dimension $d-1$ and therefore has size $q^{d-1}
%(1 + o(1))$. The set of choices for the $i^{th}$ point forms a sphere of dimension $d-i+1$. %Hence the quotient above.

\vspace{.4 cm}

Now, as argued in previous examples we have

$$\sum_{\mathbb{D}} \mu^2(\mathbb{D}) \leq \frac{1}{q^{\binom{d-k}{2}}} \sum_{g \in O(Q)}\left( \sum_z E(z)E(gz) \right)^{k+1}(1+o(1)).$$

Define $ \displaystyle f(g) = \sum_z E(z)E(gz)$ and notice, by lemma \ref{L1}, that our denominator is

$$\leq q^{-\binom{d-k}{2}}\left(|O(Q)|\left( \frac{\|f\|_1}{|O(Q)|} \right)^{k+1} + \frac{(k+1)k\|f\|_\infty^{k-1}}{2}\sum_{g} \left(f(g) - \frac{\|f\|_1}{|O(Q)|} \right)^2\right)(1+o(1))$$

Next, we have

$$\|f\|_1 = \sum_x E(x) \sum_g E(gx).$$

For each $x \in S^d$, $|\{g \in O_{d+1}(\bb{F}_q): gx = x \}|= \frac{|O(Q)|}{|S^d|}$. Thus, $\sum_g E(gx) = \frac{|O(Q)|}{|S^d|}|E|$, and consequently

$$\|f\|_1 = \frac{|O(Q)|}{|S^d|}|E|^2.$$

It is also clear that $\|f\|_\infty \leq |E|$.

\vskip.125in

Thus we have $$
\sum_{\mathbb{D}}\mu(\mathbb{D})^2 \leq q^{\binom{d+1}{2}-\binom{d-k}{2}} \left( \left( \frac{|E|^2}{|S^d|} \right)^{k+1} + \frac{(k+1)k|E|^{k-1}}{2|O_{d+1}(\bb{F}_q)|}\sum_{g} \left(f(g) - \frac{|E|^2}{|S^d|} \right)^2 \right)(1+o(1)).$$

We expand the sum inside:

$$\begin{aligned} \sum_{g} \left(f(g) - \frac{|E|^2}{|S^d|} \right)^2
& = \sum_{g} (f(g))^2 - 2\sum_{g}\frac{|E|^2}{|S^d|}f(g) + \sum_g \left( \frac{|E|^2}{|S^d|}\right)^2\\
& = \sum_{g} (f(g))^2 - 2\frac{|E|^4|O(Q)|}{|S^d|^2} +  \frac{|E|^4|O(Q)|}{|S^d|^2},\end{aligned}$$

and then observe that 

$$\begin{aligned}\sum_{g} (f(g))^2 &= \sum_g \#\{(x,y,z,w) \in E^4 : gx = y, gz = w\}\\&=S + T+R, \end{aligned}$$

where $S$ is the part of the sum coming from $4$-tuples where $x \neq \pm z$, $T$ is the part of the sum coming from
$4$-tuples where $x = z$ and $R$ is the part of the sum coming from $4$-tuples where $x=-z$.
It is easy to see that when $x=z$ we must have $y=w$ also and we get

\begin{equation} \label{T} T=\sum_g \#\{ (x,y) \in E^2 : gx = y \}  = \sum_g f(g) = \| f \|_1 = \frac{|O(Q)|}{|S^d|}|E|^2 \end{equation} 

Similarly when $x=-z$ we must have $y=-w$ also and we get

\begin{equation} \label{R} R=\sum_g \#\{ (x,y) \in (E \cap -E)^2 : gx = y \}  = \frac{|O(Q)|}{|S^d|}|E \cap -E|^2 \end{equation} 

When $x \neq \pm z$ then $y \neq \pm w$ and $x, z$ span a $2$-dimensional subspace $V$ while $y,w$ span a $2$-dimensional subspace $V'$.
$Q|V \cong Q|V'$ if and only if $\|x-z\|=\|y-w\|$ in which case the isometry extends to an element of $O(Q)$ by Witt's Theorem. Thus in this case, $\|x-z\|=\|y-w\|$
if and only if there is $g \in O(Q)$ such that $gx=y, gz=w$.

To find the stabilizer of a pair like $\{x,z \}$, fix $x$ first and note that the stabilizer of $x$ is of order $O_d(\bb{F}_q)$ and acts on the $d$-dimensional space
$x^\perp$ transitively on each sphere of $x^\perp$. Now decompose $z=kx + u$ where $k$ is some scalar and $u \in x^\perp$. Then
$Q(z) = k^2Q(x)+Q(u)+2k\langle u , x \rangle$ where $\langle , \rangle$ is the associated inner product. As $x$ and $z$ lie on the sphere of radius $r$ we see that
$Q(u)=r(1-k^2)$. Because the stabilizer of $x$ will act transitively on all $u$ of a fixed norm, we see that the stabilizer of the pair $x, z$
is the same as the $O_d(\bb{F}_q)$-stablizer of a vector $u$ in $x^\perp$. This vector $u$ is nonzero except when $z$ is a multiple of $x$ which we have avoided
by splitting up the terms $T$ and $R$.
Thus the size of the pair stabilizer will be $|O_{d-1}(\mathbb{F}_q)|$ as long as we are in a situation of ``general sphere size"
in $\mathbb{F}_q^d=x^\perp$. This happens as long as $d \geq 3$ and also when $d = 2$ as long as the corresponding circle determined by $u$ does not have
radius $0$ which only happens when $z = \pm x + u, u \in x^\perp, u \neq 0, \|u\|=0$. In this degenerate $d=2$ case, which can only happen in the hyperbolic plane, one has a stabilizer which is trivial which is $\frac{1}{2}$ of the expected stabilizer $O_1(\mathbb{F}_q)$ which has size $2$. This amounts to a factor of $2$ but will only occur for a small set of $4$-tuples of the form $(x, gx, \pm x + u, \pm y + gu)$. There are at most $|E|2(2q-1)$ of these in the $d=2$ degenerate 
case mentioned above and this discrepancy can be absorbed in the $K$ term discussed below. Thus we have

$$ S = |O_{d-1}(\mathbb{F}_q)|(1+o(1))\#\{ (x,y,z,w) \in E^4 : x \neq \pm z, \|x - z\| = \| y - w \| \}
$$
so
$$ S=|O_{d-1}(\mathbb{F}_q)|(1+o(1))\#\{ (x,y,z,w) \in E^4 : x \neq \pm z, x \cdot Az = y \cdot Aw \}, $$ 

\vskip.125in 

where $\langle x, y \rangle = x \cdot Ay$ is the inner product associated with the quadratic form $Q$. Since $Q$ is non-degenerate, $A$ is an invertible matrix. 

Now consider the quantity
$$|O_{d-1}(\mathbb{F}_q)| \cdot |\{(x,y,z,w) \in E^4: x \cdot Az = y \cdot Aw \}|$$

$$=|O_{d-1}(\mathbb{F}_q)|\sum_{t \in \mathbb{F}_q} (\nu(t))^2$$

where $\nu(t) = \{(x,y) \in E \times F: x \cdot y = t\}$, with $F=AE=\{Ax: x \in E\}$. 

\vskip.125in

It was shown in \cite{HIKR11} that

$$ \sum_{t \in \mathbb{F}_q} (\nu(t))^2 \leq \frac{|E|^2{|F|}^2}{q} + K, $$ where 
$K \leq 2|E||F|q^d$. Since $|E|=|F|$, we have 

$$\sum_{t \in \mathbb{F}_q} (\nu(t))^2 \leq \frac{|E|^4}{q} + K$$
where $K \leq 2|E|^2q^{d}$. Here we used the fact that any line through the origin intersects $E$ in at most $2$ points, as $E$ lies on a sphere
of nonzero radius.

Now we see that the terms $T$ (\ref{T}) and $R$ (\ref{R}) are no more than $|O_{d-1}(\mathbb{F}_q)|q^{d-1}|E|^2$.
Thus we see that they can be absorbed into the $S$ term by
absorbing them into the $K$ term of $\sum_{t} \nu(t)^2$ at the cost of an absolute constant $C > 2$.

Thus

$$\begin{aligned}
\sum_{g} \left(f(g) - \frac{|E|^2}{|S^d|} \right)^2 & \leq |O_{d-1}(\mathbb{F}_q)|\left(\frac{|E|^4}{q} + C|E|^2q^d\right) -\frac{|E|^4|O_{d+1}(\bb{F}_q)|}{|S^d|^2}
\\
&=\frac{|E|^4|O_{d+1}(\mathbb{F}_q)|}{|S^d|}\left(\frac{1}{q|S^{d-1}|} - \frac{1}{|S^d|}\right) + C|E|^2q^d|O_{d-1}(\mathbb{F}_q)|.
\\
& =\frac{|E|^4|O_{d+1}(\mathbb{F}_q)|}{|S^d|}\left(\frac{1}{q^d + O(q^{\frac{d+1}{2}})} - \frac{1}{q^d+O(q^{\frac{d}{2}})}\right)
+ C|E|^2q^d|O_{d-1}(\mathbb{F}_q)|.
\\
& =\frac{|E|^4|O_{d+1}(\mathbb{F}_q)|}{|S^d|}  \left(\frac{O(q^\frac{d+1}{2})}{q^{2d}}
 \right)
+ C|E|^2q^d|O_{d-1}(\mathbb{F}_q)|.
\\
&=|E|^4|O_{d-1}(\mathbb{F}_q)|O(q^{\frac{d+1}{2}-d-1})
+ C|E|^2q^d|O_{d-1}(\mathbb{F}_q)|.
\end{aligned}$$

Let $\epsilon \geq 0$ be a parameter to be optimized later. Both terms above are less than $|E|^4q^{\frac{d(d-4)+1}{2}+\epsilon}$ when $|E| \geq \sqrt{C}q^{\frac{3d+1}{4}-\frac{\epsilon}{2}}$, so our denominator in (\ref{sphereIneq}) is bounded above by

$$q^{\binom{d+1}{2}-\binom{d-k}{2}} \left( \left( \frac{|E|^2}{|S^d|} \right)^{k+1} + \frac{(k+1)k|E|^{k-1}}{2|O_{d+1}(\bb{F}_q)|}\left( 2|E|^4q^{\frac{d(d-4)+1}{2}+\epsilon} \right) \right)$$

$$ = q^{\binom{d+1}{2}-\binom{d-k}{2}} \left( \left( \frac{|E|^2}{|S^d|} \right)^{k+1} + \frac{(k+1)k|E|^{k+3}}{|S^d|q^{\frac{3d-1}{2}-\epsilon}} \right).$$

\vskip.125in

We multiply the denominator by $q^{k+1 \choose 2}$ to get

$$
q^{k+1 \choose 2} \sum_{\mathbb{D}} \mu(\mathbb{D})^2 \leq  |E|^{2k+2} + (k+1)kq^{dk-\frac{3d-1}{2}+\epsilon}|E|^{k+3} .$$

\vskip.125in

If we check when this expression is $ \leq 2|E|^{2k+2}$, we see that it suffices to have

$$|E| > (k(k+1))^{\frac{1}{k-1}} q^{d - \frac{d-1-2\epsilon}{2(k-1)}}.$$

When this happens we have $T_k^d(E) \geq \frac{1}{2} q^{k+1 \choose 2}$ as desired. Putting the two necessary conditions together we see the result holds when $|E| \geq C_0 \max(q^{d-\frac{d-1-2\epsilon}{2(k-1)}}, q^{\frac{3d+1}{4}-\frac{\epsilon}{2}})$ for suitable positive constant $C_0$ depending only on $d$ and $k$.
When $k  \geq 3$, we have $q^{\frac{3d+1}{4}} \leq q^{d-\frac{d-1}{2(k-1)}}$ and so the best choice for the optimizing parameter  is $\epsilon = 0$. When
$k=2$, the best choice of $\epsilon$ is $\epsilon=\frac{d-1}{6}$ in which case the condition becomes $|E| \geq C_0 q^{\frac{2d+1}{3}}$.
\end{proof}

\vskip.25in

\section{Proof of Theorem \ref{sharpness}}

\vskip.125in

\medskip

Part i) was established in \cite{HIKR11} for the case of $d=2$ and odd $d\geq 3$. Let us review the latter result in order to reflect on the case of even $d$ further in Part ii).
Let $q$ be a large enough prime. The construction is based on the claim in \cite{HIKR11} (Lemma 5.1) that if $d$ is even, there are $\frac{d}{2}$ mutually orthogonal linearly independent null vectors $n_i$, $i=1,\ldots,\frac{d}{2}$ in $\mathbb F_q^d$. That is $\forall i,j=1,\ldots,\frac{d}{2},$ $n_i\cdot n_j=0$. Let $N={\rm span}(n_1,\ldots, n_{\frac{d}{2}})$. Then, if $d=2k+1\geq 3$, one can take a $(k+1)$-dimensional subspace $L$, containing $N$ in  $\mathbb F_q^{2k+1},$ such that the distance form restricted to $L$ is very degenerate: for $x\in L$, $\| x\| = x^2_{k+1},$ where $x_{k+1}$ is the coordinate in $L/N.$

Hence, one can take $E=\{x\in L: x_{k+1}\in I\},$ where $I$ is an interval of length $cq$, and have $|T^d_1(E)|<2cq$, thus showing that the exponent $\frac{d+1}{2}$ in Theorem \ref{main} is optimal and attainable up to the endpoint (by Theorem \ref{main}) for the quantity $T^d_1$ and odd $d\geq 3$.

\medskip
Part ii) deals with even $d$, and for $d=2$ the usual Euclidean lattice example shows that the exponent $\alpha_{1,2}=1$ is unattainable up to the endpoint. Indeed, for any $C$, no matter how large, there is a sufficiently large prime $q$ such that the point set $E=[1,\ldots,\lceil\sqrt{Cq}\rceil]\times [1,\ldots,\lfloor\sqrt{Cq}\rfloor]$ determines $O(\frac{Cq}{\sqrt{\log q}})$ distinct distances, considered as elements of $\mathbb Z$, and therefore $O(\frac{C^2q}{\sqrt{\log q}})$ distinct distances modulo $q$. The latter quantity clearly cannot be bounded from below by $cq$, for some universal $c$ which would work for all large $q$.

A construction, joining this one with the one in the proof of Part i) above can be developed in dimension $d=2(2k+1)\geq 6$. Indeed, if $d=2(2k+1)\geq 6$, then there is a basis $\{n_1,\ldots,n_{k},e,n'_1,\ldots,n'_{k},e'\}$, with respect to which the distance form is just $x^2+y^2$, the two coordinates being relative to the length $1$ vectors $e$, $e'$.

Finally, for any even $d=2k\geq 4$  one can do the following. Let $\{n_1,\ldots,n_k\}$ span an isotropic null subspace $N$, that is $\forall i,j=1,\ldots,k$, $n_i\cdot n_j=0$. Then there is a vector $e$, such that $e\cdot e=1$, $e\cdot n_1=1$, and for $i=2,\ldots,k$, $e\cdot n_i=0$. Indeed, the vector set $\{n_1,\ldots,n_k\}$ is linearly independent. Augment it to an arbitrary basis $\{n_1,\ldots,n_k,v_1,\ldots,v_k\}$ in $\mathbb F_q^{2k}$, let $x_1,\ldots,x_k$ be the coordinates relative to the added basis vectors $\{v_1,\ldots,v_k\}$.
 Now one seeks $e= \sum_{i=1}^k x_kv_k$, and the conditions $e\cdot n_1=1$, and for $i=2,\ldots,k$, $e\cdot n_i=0$ can be satisfied if and only if the matrix 
$G$, whose elements $g_{ij}=v_i\cdot n_j$ is non-degenerate. It is, for otherwise there would be a linear combination of $v_1,\ldots,v_k$ orthogonal to $N$, but $N^\perp=N$.

Finally, since $n_1$ is defined up to a scalar multiplier, one can always ensure that $e\cdot e=1$ by scaling $n_1$ accordingly.

Hence, there is a $(k+1)$-dimensional subspace $L$ in $\mathbb F_q^{2k}$, where the distance form equals $x(x+y)$, where $x$ is the coordinate relative to $e$ and $y$ to $n_1$.
Now, let $E=\{v\in L: x,y\in [1,\ldots,\lceil\sqrt{Cq}\rceil]\}.$ Clearly $|E|\geq Cq^{k}$. On the other hand, all the distances that $E$ generates, considered as elements of $\mathbb Z$ are contained in the product set of the interval $[1,\ldots,\lceil 2\sqrt{Cq}\rceil]$, whose size is $O( \frac{Cq}{\log\log q})$ and therefore we have $O(\frac{C^2q}{\sqrt{\log q}})$ distances modulo $q$. The latter quantity clearly cannot be bounded from below by $cq$, for some universal $c$ which would work for all large $q$.

This proves Part ii).

\medskip

Next we provide a proof of Part iii) along with a separate proof of the case $k=d=2$ and the prime $q \equiv 1 \bmod 4$. The first is a general Cartesian construction. The other proof serves to  motivate the sum-product result in Corollary \ref{spcor}.

\underline{Proof 1}
Let $Q$ be a nondegenerate quadratic form on $\mathbb{F}_q^d$. Suppose, $q$ is large and write $q=p^m$ where $p$ is a prime. 

%(The forthcoming construction enables one to replace $q$ with $q^n$, with the involved constants consequently depending on $n$.)
 
First, notice that it suffices to show $\alpha_{d,d} \ge d-1 + 1/d$, as when $1 \le k < d$, we can simply use the $k=d$ case to find a suitable example in a $k$-dimensional subspace $V$ with $Q|V$ non-degenerate to get $\alpha_{k,d} \ge k-1+ 1/k$. It is easy to see such a subspace exists by the diagonalizability of quadratic forms.

Let $\epsilon \in (0, 1/d)$. Take a non-null vector $\tau$ and use the orthogonal
decomposition $\tau^{\perp} \oplus <\tau>$ to identify $\mathbb{F}_q^d = \mathbb{F}_q^{d-1} \oplus \mathbb{F}_q$. For the last 
coordinate $y$, fix a $\mathbb{F}_p$ basis for $\mathbb{F}_q$ and write $y \in \mathbb{F}_q$ as $(y(1),\dots,y(m)) \in \mathbb{F}_p^m$ where $y(i)$ 
are the $\mathbb{F}_p$ coordinates of $y \in \mathbb{F}_q$ with respect to this $\mathbb{F}_p$-basis. Consider the set

$$E =  \{ (x, y) : x \in \bb{F}_{q}^{d-1}, 0 \leq y(j) \leq  p^{1/d - \epsilon}, 1\leq j \leq m  \} \subset \bb{F}_q^{d-1} \oplus \bb{F}_q = \bb{F}_q^d.$$

Let $G$ = $$\{( (x_1, y_1), \ldots ,(x_{d+1},y_{d+1}) ) \in  (\bb{F}_q^d)^{d+1}: (x_1, y_1) \in E-E, - p^{1/d-\epsilon}  \leq y_1(j)-y_i(j) \leq p^{1/d - \epsilon} \}$$
where $1 \leq i \leq d+1, 1 \leq j \leq m$ in the description of $G$.

Notice that we can view $G$ as a set of $d$-simplices and that $E^{d+1} \subset G$. Thus the number of distinct congruence classes of $d$-tuples in $E^{d+1}$ is less than or equal to that of $G$. Also notice that if we translate a $d$-simplex in $E$ by an element in $-E$, the $d$-simplex will remain in $G$. Thus each $d$-simplex of $E$ occurs in $G$ with multiplicity at least $|E|$. We can do even better, though. Any rigid motion applied to a $d$-simplex of $G$ that preserves $y$-coordinates will keep the simplex in $G$. This includes any elements of the orthogonal group $O(Q)$ that leave the $y$-coordinate fixed, i.e., which stabilize $\tau$. There are $|O_{d-1}| = 2q^{d-1 \choose 2}(1 + o(1))$ such elements. Thus any congruence class that appears in $G$ occurs at least $2|E|q^{d-1 \choose 2}(1 + o(1))$ times. By $E$'s construction, $|E| = q^{d-1 + 1/d - \epsilon} + O(q^{d-1})$, but the number of distinct $d$-simplices up to congruence is less than

$$\frac{|G|}{2|E|q^{d-1 \choose 2}(1 + o(1))} \leq \frac{2^m|E|(2^m|E|)^d}{2|E|q^{{d-1 \choose 2}}(1+o(1))}
\leq \frac{2^{m(d+1)-1} |E|^d(1+o(1))}{q^{d-1 \choose 2}}$$
$$\leq 2^{m(d+1)-1}q^{d^2-d+1-d^2/2 + 3d/2 - 1 - d\epsilon}(1+o(1)) = 2^{m(d+1)-1}q^{{d+1 \choose 2} - d\epsilon}(1+o(1)).$$

Note $2^{m(d+1)}=p^{m log_p(2) (d+1)} = q^{log_p(2)(d+1)}$. As long as $p \to \infty$ as $q \to \infty$, then we can choose 
$\epsilon \geq 2log_p(2)$ in this argument and obtain the result. For example when $q$ is restricted to be a prime, i.e. $q=p, m=1$ then this term is independent of $q$ 
and is hence negligible as $q \to \infty$.

\underline{Proof 2}

Consider the case of prime $q\equiv 1 \;{\rm mod}\; 4$, recall that $\iota$ is the element of ${\mathbb F}_q$, such that $\iota^2=1$.

Consider the pair of ``null vectors'' $n_{\pm}$, with coordinates $(1,\pm\iota)$, respectively, in the standard basis,  as the basis in ${\mathbb F}_q^2$. Clearly, with respect to the standard dot product, $n_\pm \cdot n_\pm =0$, while $n_+\cdot n_- = 2$. Recall that $L_\pm$ are the one-dimensional ``null spaces'' spanned by, respectively, the vectors $n_\pm$. Note that
$$
L_+\cup L_- = \{\xi\in {\mathbb F}_q^2:\|\xi\|=0\},
$$
and for every $r\in {\mathbb F}_q^*$,

$$|\{\xi\in {\mathbb F}_q^2:\|\xi\|=r\}| = q-1.$$

Let us now take $E$ as the union of cosets of $L_+$ as an additive subgroup of ${\mathbb F}_q^2$, the coset representatives lying in some subset of $L_-$.

The set $E$ can be described in coordinates $(x,y)$, relative to the basis vector $\{n_+, \frac{1}{2}n_-\}$. In these coordinates, let $E= {\mathbb F}_q\times Y,$ for some $Y\subseteq {\mathbb F}_q$.

Given $u_1,u_2,u_3\in E$, with coordinates $(x_i,y_i),$ $i=1,2,3$, respectively, the distances $d_{ij}$ between pairs of distinct points, $1\leq i<j\leq 3$ equal
$$
d_{ij}  = (x_i-x_j)(y_i-y_j).
$$
Suppose, the points $u_1,u_2,u_3$ are such that all the three distances are nonzero. Then, since $x_1-x_3 = (x_1-x_2)-(x_3-x_2)$, we can eliminate the $x$ variables from the above three equations and get
\begin{equation}
\frac{d_{12}}{y_1-y_2} - \frac{d_{23}}{ y_3-y_2} = \frac{d_{13}}{ y_1-y_3}.
\label{rats}\end{equation}
Suppose $Y$ is an interval $I$ of length  $\lfloor \frac{1}{2}\sqrt{q/2-1} \rfloor$. Set $a=y_1-y_2\neq 0$ and $b=y_3-y_2\neq 0$. With $a,b\in I-I$ and $a\neq b$, we have
$$
a^{-1}d_{12} - b^{-1}d_{23} = \frac{d_{13}}{ a-b}.
$$
Given $d_{12},d_{23}\neq 0$ there are at least $\frac{q-1}{2}$ choices of $d_{13}$ so that a triangle with sidelengths $d_{12},d_{23},d_{13}$ is constructable in ${\mathbb F}_q^2$. Thus, for a fixed $d_{12}$ and $d_{23}$, the map $f:\bb{F}_q^* \times \bb{F}_q^* \rightarrow {\mathbb F}_q$, defined by
$$
f(a,b) = (1-b/a) d_{12} + (1-a/b) d_{23},
$$
cannot assume all of the possible ``third sidelength values" if $\left|\frac{I-I}{I-I}\right| < \frac{q-1}{2}$. In our case, $I$ was chosen to be an interval, so $|I-I| \leq 2|I|.$ It follows that $\left|\frac{I-I}{I-I}\right| \leq (2|I|)^2 \leq q/2-1$. $\Box$

\section{Proof of Corollary \ref{spcor}}
Suppose $q \equiv 1 \mbox{ mod }4$ is a prime. One can reverse the construction in the second proof of Theorem \ref{sharpness}. Let $E=X\times Y$, where the sets $X,Y\subseteq \mathbb F_q$ are taken along the $n_\pm$-axes. Then the set $T_1^2(E) = (X-X)\times (Y-Y)$, and the Corollary follows from (\ref{dist}). Moreover, as we have shown in the proof of Theorem \ref{main2}, the conclusion of the Theorem applies to the quantity $T_1^2(E_1,E_2)$, where the line segments in question have vertices in different sets. Hence, one can vary the signs in the ensuing sum-product type inequality by considering, say $E_1=X\times Y$ and $E_2=(-X)\times (-Y).$

The same estimate is valid in the case of prime $q \equiv 3 \mbox{ mod }4$ as well. Here, the quadratic forms $x_1^2+x_2^2$ and $x_1^2-x_2^2$ on $\mathbb F_q^2$ are different (they are the same in the $q \equiv 1 \mbox{ mod }4$ case, since one can change $x_2$ to $\iota x_2$. So suppose we are dealing with the ``Minkowski form'' $x_1^2-x_2^2$ over $\mathbb F_q^2$, $q \equiv 3 \mbox{ mod }4$.

After a change of coordinates, the Minkowski form becomes $x_1x_2$, and one can repeat the proof of Theorem \ref{main2} word by word, applying to distances
$\|u-v\|=(u_1-v_1)(u_2-v_2)$. The group $SO_2(\mathbb F_q)$ is then replaced by multiplication of elements of $\mathbb F_q^2$ by $r  \in \mathbb F_q^*$. As in the proof of Theorem \ref{main2}, one partitions the set $E$ as $E_1$ and $E_2$, now in the standard basis and looks at the Minkowski distances between two sets, when the zero distance will satisfy the estimate (\ref{nullcount}) as follows:
$$\#\{(u,v) \in E_1\times E_2: \|u-v\| = 0\}\leq 8|E|^{\frac{3}{2}}.$$
Finally, for $t\in \mathbb F_q^*$, the quantity $\sigma_E(t)$ in Lemma \ref{mlem} will be defined in exactly the same way, only relative to the hyperbola $x_1x_2=t$.
The inspection of the proof of Lemma \ref{mlem}, that is Lemma 4.4 in \cite{CEHIK12} shows that the lemma still applies to the Minkowski distance analog of the quantity $\sigma_E(t)$: the proof uses only two facts about the set $S_t=\{\xi:\|\xi\|=t\}$: that its cardinality is $q(1+o(1))$ and that the maximum multiplicity of an element of the sumset $S_t+S_t$ equals $2$.

\qed

\section{Appendix}
\subsection{Quadratic Forms, Spheres and Isotropic Subspaces}

Let $k$ be a field of characteristic not equal to $2$. Let $V$ be a finite dimensional $k$-vector space and $Q$ a quadratic form on $V$.
Associated to $Q$, there is a symmetric bilinear inner product $\langle \cdot, \cdot \rangle: V \times V \to k$ such that $Q$ is the norm map
$Q(x)=\langle x,x\rangle$. $Q$ is said to be non-degenerate if $\langle x,y \rangle = 0$ for all $y \in V$ implies $x=0$. The isometry group of
$Q$ is denoted $O(Q)$, the orthogonal group associated to $Q$, and consists of all linear transformations which preserve (the inner product associated to)
$Q$. When $Q$ is the standard dot product on $\mathbb{F}_q^d$, then the isometry group is denoted $O_d(\mathbb{F}_q)$.

$Q$ can be represented by a symmetric matrix $\mathbb{B}$ such that $Q(x)=x^T \mathbb{B} x$ and under change of basis $\mathbb{A}$,
$\mathbb{B}$ transforms as $\mathbb{B} \to \mathbb{A}^T\mathbb{B}\mathbb{A}$ and so the determinant of $\mathbb{B}$ modulo squares is an
invariant of $Q$ called the discriminant, denoted by $disc(Q)$. If $Q$ is non-degenerate, we view the discriminant as an element of $k^*/(k^*)^2$. \\

When $k$ is a finite field of odd characteristic, Witt showed (see for example \cite{Lang} or \cite{Ser73}) that two quadratic forms $Q, Q'$ are isomorphic if and only if they are equal in dimension (via the vector space they are defined on) and discriminant. Thus on $\mathbb{F}_q^d$ there are only two distinct non-degenerate quadratic forms up to isomorphism:
the standard dot product and another one. We will denote their isometry groups as $O_d(\mathbb{F}_q)$ and $O_d'(\mathbb{F}_q)$ for the remainder
of this section.

Furthermore if $x$ and $y$ are two nonzero elements of $\mathbb{F}_q^d$, Witt also showed that there is an isometry taking $x$ to $y$ if and only if
$x$ and $y$ have the same norm i.e., $Q(x)=Q(y)$. Thus for any non-degenerate quadratic form, the sphere of radius $r$ in $(\mathbb{F}_q^d, Q)$ can be identified as the quotient of $O(Q)$ by the stabilizer of any given element on that sphere. Thus a discussion of sphere sizes is equivalent to a discussion of
isometry group and stabilizer sizes.

Let $H$ denote the hyperbolic plane, the $2$-dimensional quadratic form represented by matrix
$$\begin{bmatrix} 0 & 1 \\ 1 & 0 \end{bmatrix}$$ of discriminant $-1$. $H$ has the important property that it has two totally isotropic lines (lines consisting completely of vectors of norm $0$) and its $n$-fold orthogonal direct sum $nH$ has maximal totally isotropic subspaces of dimension $n$ called Lagrangian subspaces.
$H$ plays an important role in the classification of non-degenerate quadratic forms over finite fields $F$ of odd characteristic. One can characterize completely the structure of any such non-degenerate quadratic form as follows:

If $dim(V)=2n$ is even then the isomorphism class of $Q$ has two possibilities: \\
$$
Q \cong \begin{cases}
nH \text{ if } (-1)^n disc(Q) \text{ is a square} \\
(n-1)H \oplus N_{K/F} \text{ if } (-1)^n disc(Q) \text{ is a nonsquare.}
\end{cases}
$$
Here $N_{K/F}$ is the $2$-dimensional quadratic form given by the norm map of the unique degree $2$ extension field $K$ of $F$. \\
If $dim(V)=2n+1$ is odd then
$$
Q \cong nH \oplus cx^2
$$
where $c=(-1)^n disc(Q)$.

As the standard dot product on $\mathbb{F}_q^d$ has dimension $d$ and discriminant $1$ it follows that when $d$ is even, it is isomorphic to
$\frac{d}{2} H$ if $(-1)^{\frac{d}{2}}$ is a square modulo $q$ and $\frac{d-2}{2}H \oplus N_{K/F}$ if not. When $d$ is odd, it is isomorphic to
$\frac{d-1}{2} H \oplus (-1)^{\frac{d-1}{2}}x^2$. Regardless, when $d \geq 3$, {\bf totally isotropic subspaces are unavoidable}.

The following theorem was proved by Minkowski at the age of 17 (see \cite{Min11}) based on the classification we just discussed:

\begin{theorem}
\label{thm:spheresize}
Let $S_r^Q$ denote the sphere of radius r in $(V,Q)$ i.e., $S_r^Q = \{ x \in V | Q(x)=r \}$ and let $\nu_Q(r)=|S_r^Q|$ then:
If $Q \cong nH$ so that $d=2n$ then $$\nu_Q(r)=\begin{cases} q^{d-1} - q^{\frac{d-2}{2}} \text{ if } r \neq 0 \\
q^{d-1} + q^{\frac{d}{2}} - q^\frac{d-2}{2} \text{ if } r = 0
\end{cases}$$
If $Q \cong (n-1)H \oplus N_{K/F}$ so that $d=2n$ then
$$\nu_Q(r)=\begin{cases} q^{d-1} + q^{\frac{d-2}{2}} \text{ if } r \neq 0 \\
q^{d-1} - q^{\frac{d}{2}} + q^\frac{d-2}{2} \text{ if } r = 0
\end{cases}$$
If $Q \cong nH \oplus cx^2$ so that $d=2n+1$ then
$$\nu_Q(r)=\begin{cases} q^{d-1} + q^{\frac{d-1}{2}}sgn(r/c) \text{ if } r \neq 0 \\
q^{d-1} \text{ if } r = 0
\end{cases}$$
where $sgn$ stands for the Legendre symbol. In particular the size of a sphere is always $q^{d-1}(1+o(1))$ when $d \geq 2$ except
maybe when $d=2, r=0$. The size is always also $q^{d-1} + O(q^{\frac{d}{2}})$. Also note that in any given case, there are at most 3 possibilities
for sphere sizes. Thus in any given scenario, there are only 3 possibilities for the sizes of stabilizers of nonzero points in $O(Q)$.
\end{theorem}

The degenerate case possibilities in Theorem~\ref{thm:spheresize} effect sizes of spheres and stabilizers but not the magnitude of the isometry groups.
This is because all non-degenerate quadratic forms in 1 dimension are nonzero multiples of each other and hence have exactly the same isometry group.
Then using induction in dimension $2$ and higher, and that nonzero radius sphere counts agree up to $(1+o(1))$ in all cases from Theorem~\ref{thm:spheresize}, one can prove that for any two non-degenerate $d$-dimensional
quadratic forms $Q$ and $Q'$ one has $|O(Q)| = |O(Q')|(1+o(1))$ as $q \to \infty$. In fact one has:

\begin{proposition}
\label{prop:orthogonalsize}
Let $Q$ be a non-degenerate quadratic form on $\mathbb{F}_q^d$. Then
$|O(Q)| = 2q^{{d \choose 2}}(1+o(1))$ as $q \to \infty$.
\end{proposition}
\begin{proof}
$O_d(\mathbb{F_q})/O_{d-1}(\mathbb{F}_q)$ is bijective to $S^{d-1}$, the sphere of radius $1$
in $\mathbb{F}_q^d$ by action theory. Hence
$|O_d(\mathbb{F}_q)|=|S^{d-1}| \dots |S^1||O_1(\mathbb{F}_q)|$. As $|O_1(\mathbb{F}_q)|=2$  the proposition follows from
Theorem~\ref{thm:spheresize}.
\end{proof}

\subsection{Degenerate simplices}
Let us work inside $\mathbb{F}_q^d$ where $q$ is odd in this section.

Let $[x_0, \dots, x_k]$ be a $k$-simplex in $\mathbb{F}_q^d$. After a translation this simplex transforms to the $0$-pinned $k$-simplex
$[0, x_1-x_0, \dots, x_k-x_0]$ which has the same congruence class and ordered distance vector. It is easy to see that the stabilizer in the Euclidean
group of $[x_0, \dots, x_k]$ has exactly the same size as the stabilizer in $O_d(\mathbb{F}_q)$ of the pinned simplex $[0, x_1-x_0, \dots, x_k-x_0]$ so
we will restrict out attention to $O(Q)$ and pinned simplices in this subsection.

The dimension of a $k$-simplex is the dimension of the space spanned by the vertices in its corresponding pinned simplex. A $k$-simplex is called
non-degenerate if this is as big as possible, i.e. $k$-dimensional, and degenerate otherwise. All $k$-simplices are degenerate when $k > d$.

Take a $k$-simplex and let $V$ denote the vector space generated by its associated pinned simplex. Suppose we have a non-degenerate quadratic form $Q$ on $\mathbb{F}_q^d$ like for example the standard dot product.
Let $V^\perp$ denote the perp with respect to $Q$. We then say that $V$ is good if $V \cap V^{\perp} = 0$. Under these conditions, $Q \cong Q|_V \oplus
Q|_{V^\perp}$ and it is easy to show
that the stabilizer of the $k$-simplex in $O(Q)$ can be identified with $O(Q|_{V^\perp})$ which has the size $|O_{d-m}(\mathbb{F}_q)|(1+o(1))$ where $m$ is the dimension of $V$.
Thus we have proven the following lemma:

\begin{lemma}
\label{lemma: good simplices}
Let $Q$ be a non-degenerate quadratic form on $\mathbb{F}_q^d$, $q$ odd. Consider a good $k$-simplex, i.e. a simplex whose associated pinned
simplex spans a space $V$ of dimension $m \leq k$ with $V \cap V^{\perp} = 0$. Then the stabilizer of the simplex in $O(Q)$ is $O(Q|_{V^\perp})$,
a subgroup of size
$|O_{d-m}(\mathbb{F}_q)|(1+o(1))$. The minimum size is achieved in the non-degenerate case and is of size $|O_{d-k}(\mathbb{F}_q)|(1+o(1))$.
\end{lemma}

Note that non-degenerate $d$-simplices in $\mathbb{F}_q^d$ are automatically good as $V^{\perp}=0$.

\vskip.125in

In the case the simplex is not good, in the sense that the associated spanned space has isotropic vectors, the argument of Lemma~\ref{lemma: good simplices} can break down, as Gram-Schmidt breaks down when one has vectors of norm 0. This situation will not lead to any big difficulty in this paper, though, so we will ignore it for now.
However as an example to show that things are different, consider the line $(t,it)$ in $\mathbb{F}_q^2$ where $q=1 \text{ mod } 4$ and the pinned  $1$-simplex $[(0,0), (1,i)]$. Under $O_2(\mathbb{F}_q)$, its orbit has size $2(q-1)$ which is twice as big as orbits of non-degenerate good pinned $1$-simplices. Thus the corresponding circle is twice as big as normal and so the stabilizer set is half as big as normal. Thus $(1+o(1))$ no longer suffices to record  the difference. For many arguments a factor of $2$ is no big deal but in arguments involving cancelation of lead order terms, they can play a big role.

As a final goal of this section, let us show that the set of ordered distances generated by degenerate $k$-simplices in $\mathbb{F}_q^d$ are insignificant in the set of all ordered distances of $k$-simplices.

\begin{lemma}
\label{lemma:degenerate}
Let $k \leq d$ and distances be measured with respect to any non-degenerate inner product $Q$. There are at most $O\left(k \left(q^{\frac{(k+1)k}{2}-1}\right)\right)$ ordered distances arising from degenerate $k$-simplices in $\mathbb{F}_q^d$.
\end{lemma}
\begin{proof}
It suffices to work with pinned simplices. Note that a degenerate pinned $k$-simplex lies in a vector space $V$ of dimension $<k$ which inherits an
inner product $Q|V$ from the one on $\mathbb{F}_q^d$. This inherited inner product might be degenerate. However it is known that $Q|_V=Q_0 \oplus Q_1$ where $Q_0$ is the null inner product and $Q_1$ is a
non-degenerate inner product. The ordered distances generated by our $k$-simplex would hence be identical to one lying in the vector space supporting
$Q_1$. Thus the set of ordered distances of degenerate $k$-simplices in $\mathbb{F}_q^d$ lies in the union of the set of ordered distances
of $k$-simplices in $(W,Q')$ where $(W,Q')$ range over all non-degenerate quadratic forms of dimension $m < k$.  As there are at most two distinct
non-degenerate quadratic forms per dimension, it follows that the set of ordered distances of degenerate  $k$-simplices in $\mathbb{F}_q^d$ is no
more than $2(k-1)$ times the set of ordered distances of $k$-simplices in $\mathbb{F}_q^{k-1}$ with a non-degenerate inner product.

Consider such a pinned $k$-simplex in $\mathbb{F}_q^{k-1}$. Leaving out the last vertex of the $k$-simplex yields a $(k-1)$-pinned simplex in $\mathbb{F}_q^{k-1}$ which we can assume to be non-degenerate WLOG and which determines at most $O(q^{k \choose 2})$ ordered distances as mentioned earlier.
Adding back the last simplex, we note that we have $q^{k-1}$ choices of location which each determine at most one final ordered distance vector for
each scenario of distance vector from the first $k$ vertices in that pinned $k$-simplex.

Thus there can be at most $O(q^{{k \choose 2}+(k-1)})$ ordered distances corresponding to $k$-simplices in $\mathbb{F}_q^{k-1}$. Thus the
number of ordered distances of degenerate $k$ simplices in $\mathbb{F}_q^d$ is at most $2kO(q^{{k \choose 2} + k-1}).$

As $\binom{k}{2} +(k-1) = \binom{k+1}{2}-1 < \binom{k+1}{2}$ we are done.
\end{proof}

\begin{theorem}
Let $Q$ be a non-degenerate quadratic form on $\mathbb{F}_q^d$ and let $k \leq d$. Then the number of congruence classes of non-degenerate $k$-simplices is
$O( q^{{k+1 \choose 2}} )$ as $q \to \infty$. The number of congruence classes of degenerate $k$-simplices is of smaller order and so this count is also the order of the full count of the number of congruence classes of $k$-simplices.
\label{thm: distancecount}
\end{theorem}
\begin{proof}
Two non-degenerate pinned $k$-simplices are congruent if and only if they have the same ordered distance. This is because if $\sigma$ and $\sigma'$ are two such, and $V, V'$ are the vector spaces they generate, $Q|V \cong Q|V'$ and so there is an isometry  $Q|V \to Q|V'$ which extends to an isometry in
$O(Q)$ by Witt's Theorem. Hence there are at most
$q^{{k+1 \choose 2}}$ such congruence classes as that is the maximum number of possibilities for ordered distance vectors of $k$-simplices. From lemma~\ref{lemma:degenerate}, the degenerate $k$-simplices contribute a number of congruence
classes with strictly smaller order than this, hence the theorem follows.
\end{proof}


\newpage

\begin{thebibliography}{8}

\bibitem{BIP12} M. Bennett, A. Iosevich and J. Pakianathan, {\it Three-point configurations determined by subsets of $\mathbb{F}_q^2$ via the Elekes-Sharir paradigm}, Combinatorica (to appear), (2013).

\bibitem{BKT04} J. Bourgain, N. Katz, and T. Tao {\it A sum-product estimate in finite fields, and applications} Geom. Funct. Anal. \textbf{14} (2004), 27-57.

\bibitem{BMP00} P. Brass, W. Moser and J Pach, {\it Research Problems in Discrete Geometry}, Springer (2000).

\bibitem{CEHIK12} J. Chapman, M. B. Erdogan, D. Hart, A. Iosevich and D. Koh, {\it Pinned distance sets, k-simplices, Wolff's exponent in finite fields and sum-product estimates}, Mathematische Zeitschrift, Math. Z. 271 (2012), no. 1-2, 63-93.

\bibitem{Erd05} B. Erdo\~{g}an {\it A bilinear Fourier extension theorem and applications to the distance set problem} IMRN (2006).

\bibitem{Lang} S. Lang {\it Algebra, 2nd Edition}, Addison-Wesley Publishing Co. 1984.

\bibitem{GK10} L. Guth and N. Katz, {\it On the Erd\H os distinct distance problem in the plane},  http://arxiv.org/pdf/1011.4105.

\bibitem{HI08} D. Hart and A. Iosevich, {\it Sums and products in finite fields: an integral geometric view- point}, Contemporary Mathematics: Radon transforms, geometry, and wavelets, \textbf{464}, (2008).

\bibitem{HIKR11} D. Hart, A. Iosevich, D. Koh and M. Rudnev, {\it Averages over hyperplanes, sum-product theory in finite fields, and the Erd\H os-Falconer distance conjecture}, Transactions of the AMS, \textbf{363} (2011) 3255-3275.

\bibitem{IR07} A. Iosevich and M. Rudnev {\it Erd\H os distance problem in vector spaces over finite fields}, Trans. Amer. Math. Soc. \textbf{359} (2007), no. 12, 6127-6142.

\bibitem{Mat87} P. Mattila {\it Spherical averages of Fourier transforms of measures with finite energy: dimensions of intersections and distance sets} Mathematika, \textbf{34} (1987),  207-228.

\bibitem{M95} P. Mattila, {\it Geometry of sets and measures in Euclidean spaces}, Cambridge University Press, \text{volume} 44, (1995).

\bibitem{Min11} H. Minkowski, {\it Grundlagen f\"ur eine Theorie quadratischen Formen mit ganzahligen Koeffizienten}, Gesammelte Abhandlungen, 3-145, (1911).

\bibitem{Ser73} J.P. Serre, {\it A Course in Arithmetic,} Graduate Texts in Mathematics, {\bf 7}, Springer-Verlag (1973).

\bibitem{SV04} J. Solymosi and V. Vu, {\it Distinct distances in high dimensional homogeneous sets}, Towards a theory of geometric graphs, 259-268, Contemp. Math. \textbf{342}, Amer. Math. Soc., Providence, 2004.



\end{thebibliography}
\end{document}